\documentclass[11pt]{amsart}

\usepackage{amsthm,amsmath,amssymb,geometry,listings,multirow,enumerate,textcomp}
\usepackage{color,url}
\newcommand{\btablesize}{\begin{scriptsize}}
\newcommand{\etablesize}{\end{scriptsize}}

\theoremstyle{definition} \newtheorem{cor}{Corollary}[section]
\theoremstyle{definition} \newtheorem{lem}[cor]{Lemma} 
\theoremstyle{definition} \newtheorem{prop}[cor]{Proposition} 
\theoremstyle{definition} \newtheorem{thm}[cor]{Theorem}
\theoremstyle{definition} 
\theoremstyle{definition} \newtheorem{alg}[cor]{Algorithm}

\theoremstyle{definition} \newtheorem{defn}[cor]{Definition}
\theoremstyle{definition} 
\theoremstyle{definition} 
\theoremstyle{definition}  
\theoremstyle{definition} 
\theoremstyle{definition}  
\theoremstyle{definition}  
\theoremstyle{definition}  
\theoremstyle{definition} 
\theoremstyle{definition} 
\theoremstyle{definition} \newtheorem{rmk}[cor]{Remark}
\theoremstyle{definition} 
\theoremstyle{definition} 
\theoremstyle{definition} 

\newcommand{\D}{\mathrel{\mathcal D}}
\newcommand{\N}{\mathbb N}
\newcommand{\C}{\mathbb C}
\newcommand{\G}{\mathbb G}
\newcommand{\up}{\textup}

\newcommand{\Z}{\mathbb Z}
\newcommand{\ran}{\textup{ran}}
\newcommand{\dom}{\textup{dom}}
\newcommand{\ld}{\lfloor}
\newcommand{\rd}{\rfloor}
\newcommand{\meet}{\wedge}
\newcommand{\Lev}{\textup{Lev}}

\renewcommand{\min}{m}
\renewcommand{\tt}{\texttt}

\author{Martin E.\ Malandro}

\begin{document}


\title{Enumeration of finite inverse semigroups}
\address{Department of Mathematics and Statistics, Box 2206, Sam Houston State University, Huntsville, TX 77341-2206}
\email{malandro@shsu.edu}

\begin{abstract}We give an efficient algorithm for the enumeration up to isomorphism of the inverse semigroups of order $n$, and we count the number $S(n)$ of inverse semigroups of order $n\leq 15$. This improves considerably on the previous highest-known value $S(9)$. We also give a related algorithm for the enumeration up to isomorphism of the finite inverse semigroups $S$ with a given underlying semilattice of idempotents $E$, a given restriction of Green's $\D$-relation on $S$ to $E$, and a given list of maximal subgroups of $S$ associated to the elements of $E$. 
\end{abstract}



\maketitle

\section{Introduction}
\label{SecIntro}
The development of efficient algorithms for the enumeration of finite algebraic structures dates back at least to 1955, when the first successful computer-based enumeration of the semigroups of order $n$ was accomplished, with the result that there are exactly 126 semigroups of order 4 \cite{Semi4}. The most recent result on the enumeration of finite semigroups is the 2012 result that there are exactly 12,418,001,077,381,302,684 semigroups of order 10 \cite{Semi10}. Currently, the number of semigroups of order $n$ is known only for $n\leq 10$, and there is a database of the semigroups of order 1 through 8 \cite{Smallsemi}.

This situation is in stark contrast to that for finite groups. The number of groups of order $n$ is known for $n\leq 2047$, and the {\em Small Groups Library} contains all the groups of order 2000 or less (excluding 1024), a total of 423,164,062 groups \cite{SmallGroups}.

Just as groups encode global symmetries, inverse semigroups encode partial symmetries \cite{Lawson}. Every group is an inverse semigroup, and every inverse semigroup is a semigroup, but neither conversely. 
This paper marks the first attempt to enumerate finite inverse semigroups specifically. If we denote by $S(n)$ the number of inverse semigroups of order $n$ up to isomorphism, then the numbers $S(1) ,\ldots, S(9)$ are known from previous work on the enumeration of finite semigroups. $S(9)$ was computed by A.\ Distler in his 2010 Ph.D.\ thesis \cite{Semi9,DistlerThesis}. Previously $S(8)$ was computed by S.\ Satoh, K.\ Yama, and M.\ Tokizawa in 1994 \cite{Semi8}. The references in \cite{Semi8} contain information concerning the history of the computation of $S(n)$ for $n\leq 7$.
In 2012 Distler et al.\ found the number of semigroups of order 10
with the help of a parallelized computation which took approximately 130 CPU years 
\cite{Semi10}. Although it may have been possible to compute $S(10)$ along the way during this computation, inverse semigroups are not discussed and $S(10)$ is not reported in \cite{Semi10}. At present there is no explicit formula for $S(n)$, and the only way to compute $S(n)$ is by a careful exhaustive search.

The approach to semigroup enumeration in \cite{Semi10,Semi9,DistlerThesis} is based on the idea that any combinatorial enumeration problem can be written as a constraint satisfaction problem. To obtain their results in \cite{Semi10}, the authors work out a collection of constraint satisfaction problems whose solutions comprise the semigroups of order $n$ which cannot be counted by any previously-known formula and feed those problems to the constraint satisfaction solver \texttt{minion} \cite{minion}. 
The constraint satisfaction approach can be adapted to count inverse semigroups by adding additional constraints to be satisfied. Using this idea and the code from \cite{DistlerThesis} we were able to compute $S(10)$ in 11.5 hours of CPU time on a test machine with an Intel\textregistered \textup{ }Core\textsuperscript{TM} i3 processor and 8 GB of RAM. This computation consumed about 7 GB of RAM. We were unable to go beyond $S(10)$ on our test machine with this approach. 

In this paper we offer a specialized method for the enumeration of the inverse semigroups of order $n$ (Algorithm \ref{AlgMain}), which is both quicker and more memory-efficient than the generic approach of \cite{Semi10,Semi9,DistlerThesis}. We have implemented our algorithm in \texttt{Sage}, an open-source computer algebra system \cite{sage}. For comparison, our implementation computes $S(10)$ on our test machine in 22 minutes of CPU time using only 256 MB of RAM. We have used our implementation to enumerate the inverse semigroups of order $1 ,\ldots, 15$. Along the way we also counted commutative inverse semigroups, inverse monoids, and commutative inverse monoids. Our enumeration results are summarized in Table \ref{TableISGs}. More detailed counts and information regarding the efficiency of our algorithm are given in Section \ref{SecResults}.

\begin{table}[ht]
\btablesize
\caption{The inverse semigroups of order $\leq 15$}
\begin{tabular}{|c|c|c|c|c|}
\hline
\multirow{2}{*}{$n$} & \# inverse semigroups & \# commutative inverse & \# inverse monoids & \# commutative inverse\\
&of order $n$& semigroups of order $n$ & of order $n$ & monoids of order $n$ \\
\hline
1 & 1 & 1 & 1 & 1\\
\hline 2 & 2 & 2 & 2 & 2 \\
\hline 3 & 5 & 5 & 4 & 4 \\
\hline 4 & 16 & 16 & 11 & 11 \\
\hline 5 & 52 & 51 & 27 & 27 \\
\hline 6 & 208 & 201 & 89 & 87 \\
\hline 7 & 911 & 877 & 310 & 300 \\
\hline 8 & 4637 & 4443 & 1311 & 1259 \\
\hline 9 & 26422 & 25284 & 6253 & 5988 \\
\hline 10 & 169163 & 161698 & 34325 & 32812 \\
\hline 11 & 1198651 & 1145508 & 212247 & 202784 \\
\hline 12 & 9324047 & 8910291 & 1466180 & 1400541 \\
\hline 13 & 78860687 & 75373563 & 11167987 & 10669344 \\
\hline 14 & 719606005 & 687950735 & 92889294 & 88761928 \\
\hline 15 & 7035514642 & 6727985390 & 836021796 & 799112310 \\
\hline
\end{tabular}
\label{TableISGs}
\etablesize
\end{table}

We have stored and made available the Cayley tables of the inverse semigroups of order $2 ,\ldots, 12$, as well as the \texttt{Sage} code we used to count the inverse semigroups of order $\leq 15$\footnote{Cayley tables and \tt{Sage} code available at \url{http://www.shsu.edu/mem037/ISGs.html}.}. A key step of Algorithm \ref{AlgMain} involves iterating over the unlabeled meet-semilattices of order $1 ,\ldots, n$ up to isomorphism. For this step we used the algorithm of Heitzig and Reinhold \cite{CountingFiniteLattices} (see also \cite{LatticeEnum2}) for generating finite lattices. For $n\leq 15$ we have stored and made available the lattices of order $n$ as lists of cover relations\footnote{Lattices available at \url{http://www.shsu.edu/mem037/Lattices.html}.}.


Our approach to inverse semigroup enumeration is based on the Ehresmann-Schein-Nambooripad (ESN) theorem (see, e.g., \cite[Ch.\ 4]{Lawson} and \cite{E,N,S}), which essentially transfers the problem of enumerating inverse semigroups to the problem of enumerating a certain class of groupoids. Briefly, let $r_1 ,\ldots, r_k$ be positive integers and let $G_1 ,\ldots, G_k$ be finite groups. Let $A$ denote the algebra 
\[
A=\bigoplus_{i=1}^k M_{r_i}(\C G_i).
\]
Let $B$ be the collection of matrices in $A$ which have a group element in exactly one position and 0 elsewhere. $B$ is called the {\em natural basis} of $A$. The set $E(B)$ of idempotents of $B$ is the set of matrices having a group identity on the diagonal. If $\leq$ is a partial order on $B$ for which $(E(B),\leq)$ is a meet-semilattice and which satisfies some additional properties (see Theorem \ref{ThmESN}), then the ESN theorem may be used to construct from $(B,\leq)$ an inverse semigroup $S$ such that
\[
|S| = \sum_{i=1}^k {r_i}^2 |G_i|.
\]
Furthermore, to generate all inverse semigroups $S$ of order $n$, it suffices to apply the construction from the ESN theorem to all partial orders $\leq$ on all natural bases $B$ satisfying the hypotheses of Theorem \ref{ThmESN}, across all algebras $A=\bigoplus_{i=1}^k M_{r_i}(\C G_i)$ for which $n = \sum_{i=1}^k {r_i}^2 |G_i|.$

Unfortunately, this process may create isomorphic inverse semigroups. To generate the inverse semigroups of order $n$ up to isomorphism, we generate inverse semigroups according to this process, testing each newly generated inverse semigroup for isomorphism against previously-generated inverse semigroups, and accepting only the inverse semigroups not isomorphic to any previously-generated inverse semigroup. To accomplish our isomorphism testing in an efficient manner we begin by separating generated inverse semigroups according to invariants. Any function $I$ such that whenever $S,T$ are inverse semigroups with $S\cong T$ we have $I(S)=I(T)$ is called an {\em invariant}. A newly-generated inverse semigroup $S$ must be tested for isomorphism only against previously-generated inverse semigroups $T$ such that $I(S)=I(T)$. 

The main computational challenges we face are the challenges of efficiently generating the partial orders on $B$, addressed in Section \ref{SecPartialOrders}, giving an effective definition for $I$, addressed in Section \ref{SecInvariants}, and quickly testing for isomorphism between pairs of inverse semigroups, addressed in Section \ref{SecIsNew}. A complete description of our enumeration procedure is given as Algorithm \ref{AlgMain} in Section \ref{SecAlgms}. Our algorithm adopts a memory-efficient iteration order---only a relatively small percentage of the inverse semigroups of order $n$ need to be held in memory for isomorphism testing at any given point during our algorithm's execution.

We also give a modification of our algorithm that enumerates the inverse semigroups $S$ having a particular semilattice $E(S)$ of idempotents, a particular restriction of Green's $D$-relation on $S$ to $E(S)$, and a particular collection of maximal subgroups of $S$ associated to the elements of $E(S)$. This modification is given as Algorithm \ref{AlgBasic} in Section \ref{SecAlgms}.




\section{Inverse semigroups}

\label{SecBasics}

In this section we collect the basic definitions and facts about inverse semigroups we require to explain our algorithms and prove their correctness. Good references for the facts in this section include \cite{CliffPres,Green,Lawson,Steinberg2}. We recount the basic theory of inverse semigroups in Section \ref{SubSecBasics}, we discuss facts about inverse semigroup isomorphisms in Section \ref{SubsecIsom}, and we examine the restriction of Green's $\D$-relation on $S$ to $E(S)$ in Section \ref{SubSecGreens}. In this paper we write our semigroup operations multiplicatively.

\subsection{Inverse semigroup basics}



\label{SubSecBasics}

\begin{defn}An {\em inverse semigroup} is a nonempty semigroup $S$ where, for each $x\in S$, there exists a unique $y\in S$ such that $xyx=x$ and $yxy=y$. We call $y$ the {\em inverse} of $x$, and we write $x^{-1}=y$.
\end{defn}

There is a well-known alternate characterization of inverse semigroups.

\begin{thm}A nonempty semigroup $S$ is an inverse semigroup if and only if $S$ is regular (meaning for each $x\in S$ there exists $y\in S$ such that $xyx=x$ and $yxy=y$) and the idempotents of $S$ commute.
\end{thm}

Let $S$ be an inverse semigroup. We denote by $E(S)$ the set of idempotents of $S$. If $x\in S$, it is immediate that $xx^{-1},x^{-1}x\in E(S)$, so $|E(S)|\geq 1$, and furthermore that if $e\in E(S)$, then $e=e^{-1}$. It is also clear from the alternate characterization that if $e,f\in E(S)$, then $ef\in E(S)$. It is also well known and easy to prove that $S$ is a group if and only if $|E(S)|=1$. 

\begin{defn}Let $S,T$ be semigroups. A map $\phi:S\rightarrow T$ is called a {\em homomorphism} if $\phi(st)=\phi(s)\phi(t)$ for all $s,t\in S$, and is called an {\em anti-homomorphism} if $\phi(st)=\phi(t)\phi(s)$ for all $s,t\in S$. An {\em isomorphism} is a bijective homomorphism and an {\em anti-isomorphism} is a bijective anti-homomorphism. $S$ and $T$ are {\em isomorphic} (resp., {\em anti-isomorphic}) if there is an isomorphism (resp., anti-isomorphism) from $S$ to $T$. If $S$ and $T$ are isomorphic, we write $S\cong T$.
\end{defn}

We now recall the natural partial order on $S$. 

\begin{defn}For $s,t\in S$, we define $s\leq t$ if and only if any of the following four equivalent conditions hold.
\begin{itemize}
	\item $s=et$ for some $e\in E(S)$. 
	\item $s=tf$ for some $f\in E(S)$.
	\item $s=ts^{-1}s$.
	\item $s=ss^{-1}t$.
\end{itemize}
\end{defn}

Thus, the restriction of $\leq$ to $E(S)$ is given by the rule that, for $e,f\in E(S)$, $e\leq f$ if and only if $e=ef=fe$. Therefore $(E(S),\leq)$ is a meet-semilattice, where the meet $e\meet f$ of $e,f\in E(S)$ is given by $e\meet f=ef=fe$. Conversely, every meet-semilattice is an idempotent inverse semigroup under the meet operation. 

As is common in inverse semigroup theory, for $s\in S$, let us write $$\dom(s)=s^{-1}s$$  and  $$\ran(s)=ss^{-1}.$$

Next we recall Green's $\D$-relation, which takes on a particularly nice form for inverse semigroups. 

\begin{defn}For $s,t\in S$, we say that $s$ is $\D$-related to $t$ and we write $s \D t$ if and only if any of the following equivalent conditions hold.
\begin{enumerate}[(i)]
	\item \label{DEquiv1} There exists $x\in S$ such that $\dom(x)=\ran(t)$ and $\ran(x)=\ran(s)$.
	\item \label{DEquiv2} There exists $y\in S$ such that $\dom(y)=\dom(t)$ and $\ran(y)=\dom(s)$.
	\item \label{DEquiv3} There exists $z\in S$ such that $\dom(z)=\dom(t)$ and $\ran(z)=\ran(s)$.
\end{enumerate} 
\end{defn}
To see that these three conditions are equivalent, for $\eqref{DEquiv1}\implies \eqref{DEquiv2},$ take $y=s^{-1}xt$, for $\eqref{DEquiv2}\implies \eqref{DEquiv3},$ take $z=sy$, and for $\eqref{DEquiv3}\implies \eqref{DEquiv1},$ take $x=zt^{-1}$.

$\D$ is an equivalence relation on $S$, and the equivalence classes of $S$ under $\D$ are called the $\D$-classes of $S$. For finite inverse semigroups, an equivalent characterization of $\D$ is that, for $s,t\in S$, $s\D t$ if and only if $s$ and $t$ generate the same two-sided ideal in $S$. 

Next we recall the notion of a maximal subgroup of an inverse semigroup.
\begin{defn}A subset $G$ of $S$ is a {\em subgroup} of $S$ if $G$ is a group with respect to the operation of $S$. A subgroup $G$ of $S$ is {\em maximal} if $G$ is not properly contained in any other subgroup of $S$. 
\end{defn}

Every idempotent $e$ of $S$ is the identity for a unique maximal subgroup of $S$, called the {\em maximal subgroup of $S$ at $e$} and denoted $G_e$. For any subgroup $G$ of $S$ containing $e$, we have $G\subseteq G_e$. In fact 
$$
G_e=\{ s\in S:\dom(s)=\ran(s)=e\},
$$
and if $e,f\in E(S)$ with $e \D f$, then $G_e\cong G_f$. Let $P$ denote the restriction of Green's $\D$-relation on $S$ to $E(S)$. Every $\D$-class of $S$ contains at least one idempotent, so $P$ is a partition of $E(S)$. Let $X\in P$ and let $G$ be a group. If $G_e\cong G$ for some $e\in X$ (and hence for every $e\in X$), then we shall write $G_X \cong G$.

We now record several important basic properties. 

\begin{thm}Let $s,t,y,z\in S$ and let $e,f\in E(S)$.
\label{ThmBasicProps}
\begin{enumerate}[(i)]
	\item If $s\leq e$, then $s\in E(S)$.  \label{simple6}

\item $s^{-1} \D s \D \dom(s) \D \ran(s)$. \label{simplecombo}

	\item If $s$ and $t$ are in the same maximal subgroup, then $s\D t$.\label{simple5} 
	\item $ses^{-1}$ is idempotent. \label{simple7}
	\item If $s\leq t$, then $s^{-1}\leq t^{-1}$. \label{simple10}
	\item If $s\leq y$ and $t\leq z$, then $s t\leq y z$.
	\item If $e\leq \dom(s)$, then $t=se$ is the unique element of $S$ such that $t\leq s$ and $\dom(t)=e$.
	\item If $e\leq \ran(s)$, then $t=es$ is the unique element of $S$ such that $t\leq s$ and $\ran(t)=e$. \label{simple11}
\end{enumerate}
\end{thm}

\subsection{Facts about isomorphisms}
\label{SubsecIsom}
Results on the enumeration of semigroups are typically reported up to {\em equivalence} (meaning isomorphism or anti-isomorphism) \cite{Semi10,Semi9, DistlerThesis, Semi8}, although some enumeration results have also been completed and reported up to isomorphism \cite{Semi9, DistlerThesis}. For inverse semigroups, these concepts agree. 

\begin{thm}Let $S$ and $T$ be inverse semigroups. Then $S$ and $T$ are isomorphic if and only if $S$ and $T$ are anti-isomorphic.
\end{thm}
\begin{proof}
Define $i_T:T\rightarrow T$ by $i_T(t)=t^{-1}$ for $t\in T$. If $\phi:S\rightarrow T$ is an isomorphism, then it is easy to check that $i_T\circ\phi$ is an anti-isomorphism. Similarly, if $\phi:S\rightarrow T$ is an anti-isomorphism, then $i_T \circ \phi$ is an isomorphism.
\end{proof}

Isomorphisms of inverse semigroups preserve their substructures. In particular, it is straightforward to verify the following result. Let $S$ and $T$ be finite inverse semigroups.

\begin{thm}
Suppose $E=E(S)=E(T)$ and $\phi: S\rightarrow T$ is an isomorphism. Then:
\label{ThmIsomTest}
\begin{enumerate}[(i)]
	\item \label{Isom0} $e\in E$ if and only if $\phi(e)\in E$.
	\item\label{Isom1} If $G_e$ is any maximal subgroup of $S$, then $\phi(G_e)$ is a maximal subgroup of $T$ and $\phi$ restricted to $G_e$ is an isomorphism of groups. 
	\item\label{Isom2} $\phi$ restricted to $E$ is a poset automorphism.
	\item \label{ThmPosetIsom1}$\phi$ is a poset isomorphism $\phi: (S\leq) \rightarrow (T,\leq)$. 
	\item\label{Isom3} $\phi$ restricts to bijections between the $\D$-classes of $S$ and $T$. In particular, let $D_1 ,\ldots, D_{k_1}$ be the $\D$-classes of $S$ and let $F_1 ,\ldots, F_{k_2}$ be the $\D$-classes of $T$. Then $k_1=k_2$, and for $x,y\in S$, $x\D y$ if and only if $\phi(x)\D \phi(y)$.
	\item \label{Isom5} If $D$ is a $\D$-class of $S$ with $k$ idempotents, then $\phi(D)$ has $k$ idempotents.
	\item \label{Isom4} For all $x\in S$, $\phi(\ran(x)) = \ran(\phi(x))$ and $\phi(\dom(x)) = \dom(\phi(x))$. 
\end{enumerate}
\end{thm}

We now examine the action of isomorphisms on a particular type of idempotent, which we have chosen to call {\em lonely}. Our main result about lonely idempotents is Theorem \ref{ThmLonely}, which states that if $S \cong T$, then any set map from the lonely idempotents of $S$ to the lonely idempotents of $T$ may be extended to an isomorphism.

\begin{defn}Let $e\in E(S)$. We say $e$ is a {\em lonely idempotent} if
\begin{itemize}
	\item $G_e \cong \Z_1$,
	\item $\{e\}$ is a $\D$-class of $S$, 
	\item $e$ covers the minimal element of $E(S)$, and
	\item $e$ is not covered by any element of $E(S)$.
\end{itemize}
\end{defn}

\begin{lem}
\label{LemLonelyMult}
Let $s,t\in S$. Let $\min$ denote the minimal element of $E(S)$. 
\begin{enumerate}[(i)]
	\item\label{PropLonelyMultPt1} If $s$ or $t$ is not a lonely idempotent, then $s t$ is not a lonely idempotent.
	\item\label{PropLonelyMultPt2} If $s$ and $t$ are lonely idempotents with $s\neq t$, then $s t = \min$.
\end{enumerate}
\end{lem}

\begin{proof}We begin by proving the contrapositive of part \eqref{PropLonelyMultPt1}. Suppose $s t$ is a lonely idempotent. Write $s t=e$ for some $e\in E(S)$. Then $e=s t = s s^{-1} s t = s s^{-1} e \leq s s^{-1}$. Since $e$ is not covered by any element of $E(S)$, $e$ is not covered by any element of $S$. After all, if $e<x$ for some $x\in S$, then $e<x^{-1}$, so $e<xx^{-1}$---but $xx^{-1}\in E(S)$. Thus $e=ss^{-1}$. It follows that $e\D s$, but since $\{e\}$ is a $\D$-class, we have $e=s$. Next, $e=s t = e t \leq t$, so $e=t$. Thus, $s$ and $t$ are both lonely idempotents.

For part \eqref{PropLonelyMultPt2}, suppose $s$ and $t$ are lonely idempotents with $s\neq t$. Since $s$ and $t$ cover $\min$ and $s t = s\meet t$, the meet of $s$ and $t$ in $E(S)$, we have $s t = \min$.
\end{proof}

\begin{thm}
\label{ThmLonely}
Suppose $\phi:S \rightarrow T$ is an isomorphism, the set of lonely idempotents of $S$ is $L(S)=\{e_1 ,\ldots, e_{r_1}\}$, and the set of lonely idempotents of $T$ is $\{f_1 ,\ldots, f_{r_2}\}$. Then $r_1 = r_2$, and if we define $\gamma: S \rightarrow T$ by 
\[
\begin{aligned}
&\gamma(e_i) = f_i && \up{for } e_i \in L(S);\\
&\gamma(x) = \phi(x) && \up{for } x\notin L(S),
\end{aligned}
\]
then $\gamma$ is an isomorphism.
\end{thm}

\begin{proof}
Let $\phi:S\rightarrow T$ be an isomorphism. 

Suppose $e\in S$ is a lonely idempotent. Then $\{\phi(e)\}$ is a $\D$-class of $T$ (since $\phi$ maps $\D$-classes to $\D$-classes),  $G_\phi(e) \cong \Z_1$ (since $\phi$ restricts to a bijection between maximal subgroups),  $\phi(e)$ covers the minimal element of $E(T)$ (since $\phi$ is a poset isomorphism), and $\phi(e)$ is not covered by any element of $E(T)$ (since $\phi$ is a poset isomorphism). That is, $\phi(e)$ is a lonely idempotent. Since $\phi$ is an isomorphism, the same argument applied to $\phi^{-1}$ shows that $e\in S$ is a lonely idempotent if and only if $\phi(e)\in T$ is a lonely idempotent and hence $r_1=r_2$. 

Let $r=r_1=r_2$ and let $\sigma$ be the permutation of $\{1 ,\ldots, r\}$ such that $\phi(e_i) = f_{\sigma(i)}$ for all $i\in \{1 ,\ldots, r\}$. It is clear that $\gamma$ is a bijection. Let $s,t\in S$. We must show $\gamma(s t) = \gamma(s)\gamma(t)$. 
\begin{itemize}
	\item[Case 1.] Suppose neither $s$ nor $t$ is a lonely idempotent. Then $s t$ is not a lonely idempotent, and $\gamma(s t) = \phi(s t) = \phi(s)\phi(t) = \gamma(s)\gamma(t)$. 
	\item[Case 2.] Suppose exactly one of $s,t$ is a lonely idempotent. Then $s t$ is not a lonely idempotent. Let $\min$ denote the minimal element of $E(T)$.
	
	Suppose $x\in T$ is any lonely idempotent and $y\in T$ is not a lonely idempotent. Then $yy^{-1}$ is not a lonely idempotent, so we have $\min=xyy^{-1}$, so $\min y = xy$. Similarly we have $y\min = yx$. 
	
	If $s$ is a lonely idempotent, then so is $\gamma(s)$, while $\gamma(t)$ is not. Say $s=e_i$. We have $\gamma(s t)=\phi(s t) = \phi(s)\phi(t) = f_{\sigma(i)} \gamma(t) = \min \gamma(t) = f_i \gamma(t) = \gamma(s)\gamma(t)$. 
	On the other hand, if $t$ is a lonely idempotent, then so is $\gamma(t)$, while $\gamma(s)$ is not. Say $t=e_i$. Then $\gamma(s t) = \phi(s t) = \phi(s) \phi(t) = \gamma(s) f_{\sigma(i)} = \gamma(s)\min = \gamma(s)f_i = \gamma(s)\gamma(t)$.
	
	\item[Case 3.] Suppose $s$ and $t$ are both lonely idempotents. Then so are $\phi(s)$, $\phi(t)$, $\gamma(s)$, and $\gamma(t)$. Let $m_S$ and $m_T$ denote the minimal elements of $S$ and $T$, respectively.	Say $s=e_i$ and $t=e_j$. If $s=t$, then $\gamma(s t) = \gamma(s) = \gamma(e_i) = f_i = \phi(e_{\sigma^{-1}(i)}) = \phi(e_{\sigma^{-1}(i)}e_{\sigma^{-1}(i)}) = \phi(e_{\sigma^{-1}(i)})\phi(e_{\sigma^{-1}(i)}) = \gamma(e_i) \gamma(e_i) = \gamma(s)\gamma(t)$. On the other hand, if $s\neq t$, then $s t=\min_S$, and $\gamma(s)\neq \gamma(t)$, so $\gamma(s)\gamma(t) = \min_T$. Therefore $\gamma(s t) = \gamma(\min_S) = \phi(\min_S) = \min_T = \gamma(s)\gamma(t)$. 
\end{itemize}
\end{proof}

\subsection{The restriction of Green's $\D$-relation to $E(S)$}
\label{SubSecGreens}
Let $E$ be a meet-semilattice and $P$ a set partition of $E$. Write $\sim$ for the equivalence relation on $E$ induced by $P$. That is, for $e,f\in E$, $e\sim f$ if and only if $e$ and $f$ are in the same part of $P$.

\begin{defn}We say $P$ is a {\em $\D$-partition} of $E$ if, whenever $e_1,e_2,f\in E$ with $e_1\sim e_2$ and $f\leq e_1$, we have
\[
|h\in E: h\leq e_1 \up{ and }h \sim f| = |h\in E : h\leq e_2 \up{ and } h \sim f|.
\]
\end{defn}

The main result of this section is Theorem \ref{ThmLegitPartition}, which states that any set partition $P$ of $E$ for which there exists an inverse semigroup $S$ such that $E(S)=E$ and such that the restriction of $\D$ from $S$ to $E(S)$ is $P$ is a $\D$-partition of $E$. In fact we have the following more general result. Let $S$ be an inverse semigroup.

\begin{prop}
\label{PropLegitPartition}
Let $s,t\in S$, let $e\in E(S)$, and let $s\D e$. Let $f\leq e$. 
	Then $$|\{t\in S:t\leq s \up{ and }t\D f\}| = |\{g\in S:g\leq e \up{ and }g\D f\}|.$$
\end{prop}
\begin{proof}
 Let $f\leq e$, so $f$ is idempotent. Let $x\in S$ such that $xx^{-1}=e$ and $x^{-1}x=ss^{-1}$. Define $$\psi:\{g\in S:g\leq e \up{ and }g\D f\}\rightarrow\{t\in S:t\leq s \up{ and } t\D f\}$$ by $\psi(g)=x^{-1}gxs.$ We claim that $\psi$ is a bijection. 

First, to show that the codomain of $\psi$ really is 
$\{t:t\leq s \up{ and } t\D f\}$, let $g\leq e$ (so $g$ is idempotent and $g=eg=ge$) and $g\D f$. Note that $x^{-1}gx$ is idempotent (by part \eqref{simple7} of Theorem \ref{ThmBasicProps}), so $\psi(g)\leq s$. We need to show that $\psi(g)\D f$. We have that
\begin{align*}
\ran(\psi(g))&=\psi(g)\psi(g)^{-1}\\
&=x^{-1}gx(ss^{-1})x^{-1}gx\\
&=x^{-1}gx(x^{-1}x)x^{-1}gx\\
&=x^{-1}g(xx^{-1})(xx^{-1})gx\\
&=x^{-1}(xx^{-1})(xx^{-1})ggx\\
&=x^{-1}gx,
\end{align*}
so we need to show that there exists $z\in S$ such that $zz^{-1}=x^{-1}gx$ and $z^{-1}z=f$. Since $g\D f$, let $y\in S$ such that $yy^{-1}=g$ and $y^{-1}y=f$. Let $z=x^{-1}y$.  
Then
$
zz^{-1}=x^{-1}yy^{-1}x=x^{-1}gx,
$
as desired. Note that $y=yy^{-1}y=gy$, so we also have
$$
z^{-1}z =y^{-1}xx^{-1}y
=y^{-1}ey
=y^{-1}egy
=y^{-1}gy
=y^{-1}y
=f,
$$
so $z^{-1}z=f$, as desired. Thus $\psi(g)\D f$ and the codomain of $\psi$ is as claimed.

Next, to show $\psi$ is injective, suppose $g_1,g_2\leq e$ (so $eg_1=g_1$ and $eg_2=g_2$) and $\psi(g_1)=\psi(g_2)$. Then 
\begin{align*}
x^{-1}g_1xs&=x^{-1}g_2xs\\
\implies x^{-1}g_1xss^{-1}&=x^{-1}g_2xss^{-1}\\
\implies x^{-1}g_1xx^{-1}x&=x^{-1}g_2xx^{-1}x\\
\implies x^{-1}g_1x&=x^{-1}g_2x\\
\implies xx^{-1}g_1xx^{-1}&=xx^{-1}g_2xx^{-1}\\
\implies eg_1e&=eg_2e\\
\implies eeg_1&=eeg_2\\
\implies eg_1&=eg_2\\
\implies g_1&=g_2.
\end{align*}

Finally, to show that $\psi$ is surjective, let $t\leq s$ and $t\D f$. Then $t=gs$ for some $g\in E(S)$, and there exists $y\in S$ such that $yy^{-1}=tt^{-1}$ and $y^{-1}y=f$. We claim that $xgx^{-1}\leq e$, $xgx^{-1}\D f$, and  $\psi(xgx^{-1})=t$. To see why, first note that $xgx^{-1}$ is idempotent and 
$xgx^{-1} =xgx^{-1}xx^{-1} =xgx^{-1}e$,
so $xgx^{-1}\leq e$. Next, we show that $xgx^{-1}\D f$ by showing that there exists $z\in S$ such that $zz^{-1}=xgx^{-1}$ and $z^{-1}z=f$. If we take $z=xy$, then we have
\begin{align*}
zz^{-1} &=xyy^{-1}x^{-1}\\
&=x tt^{-1} x^{-1}\\
&=x gss^{-1}g x^{-1}\\
&=x g x^{-1}x g x^{-1}\\
&=xggx^{-1}xx^{-1}\\
&=xgx^{-1}.
\end{align*}
Note that $tt^{-1}=gss^{-1}g=gx^{-1}xg=gx^{-1}x$, which implies that 
$$y=yy^{-1}y=tt^{-1}y=gx^{-1}xy.$$ Note also that $y^{-1}=y^{-1}yy^{-1}=fy^{-1}$. Therefore we  have
\begin{align*}
z^{-1}z&=y^{-1}x^{-1}xy\\
&=fy^{-1}x^{-1}xy\\
&=fy^{-1}x^{-1}xgx^{-1}xy\\
&=fy^{-1}gx^{-1}xx^{-1}xy\\
&=fy^{-1}gx^{-1}xy\\
&=fy^{-1}y\\
&=ff\\
&=f,
\end{align*}
so $xgx^{-1} \D f$. Finally we show that $\psi(xgx^{-1})=t$. We have 
$\psi(xgx^{-1}) = x^{-1} xgx^{-1} xs
=gx^{-1}xx^{-1}xs
=gss^{-1}ss^{-1}s
=gs =t.$
Thus $\psi$ is surjective, completing the proof.
\end{proof}

Combining Proposition \ref{PropLegitPartition} and part \eqref{simple6} of Theorem \ref{ThmBasicProps}, we obtain the following result.

\begin{thm}
\label{ThmLegitPartition}
Let $e_1, e_2, \in E(S)$ with $e_1 \D e_2$. Suppose $f\in E(S)$ with $f\leq e_1$. Then
\[
|h\in E(S):h\leq e_1 \up{ and }h\D f| = |h\in E(S):h\leq e_2 \up{ and }h\D f|.
\]
\end{thm}

\section{Enumeration via the Ehresmann-Schein-Nambooripad theorem}

\label{SecESN}

\subsection{The ESN theorem}

\label{SubSecESN}

The Ehresmann-Schein-Nambooripad (ESN) theorem provides an isomorphism between the category of inverse semigroups and homomorphisms and the category of inductive groupoids and inductive functors. See, e.g., \cite[Ch.\ 4]{Lawson} and \cite{E,N,S}. In this section we review the portion of the ESN theorem we need through the lens of B.\ Steinberg's construction of an isomorphism between the algebra of a finite inverse semigroup and a direct sum of matrix algebras over group algebras \cite{Steinberg2}. 

Let $r_1 ,\ldots, r_k$ be positive integers and let $G_1 ,\ldots, G_k$ be finite groups. Let $A$ denote the algebra 
\[
A = \bigoplus_{i=1}^k M_{r_i}(\C G_i).
\]
The {\em natural basis} $B$ of $A$ is the set of matrices in $A$ having a single group element in one position and $0$ elsewhere. Let $E(B)$ denote the set of idempotents of $B$. It is easy to see that $E(B)$ consists precisely of the elements of $B$ having a group identity on the diagonal.
Even though $B$ is not generally an inverse semigroup, for $b\in B$ let $b^{-1}$ denote the matrix obtained by taking the transpose of $b$ and replacing the group element in $b$ with its (group) inverse. It is easy to see that $bb^{-1}b=b$ and $b^{-1}bb^{-1}=b^{-1}$, so $bb^{-1}$ and $b^{-1}b$ are idempotent.
We continue to use the notation 
\[
\dom(b) = b^{-1}b
\]
and
\[
\ran(b) = bb^{-1},
\]
so for $a,b\in B$, $ab$ is nonzero if and only if $\dom(a)=\ran(b)$. 


First we review how the ESN theorem allows us to construct inverse semigroups from $B$.

\begin{thm}[ESN theorem pt.\ 1]
\label{ThmESN}
Let $B$ be the natural basis of the algebra $\bigoplus_{i=1}^k M_{r_i}(\C G_i)$ and let $0$ denote the zero matrix. Let $\leq$ be any partial order on $B$ satisfying the following properties.
\begin{enumerate}[(i)]
	\item $\leq$ restricted to $E(B)$ forms a meet-semilattice. \label{ESNcondition1}
	\item $\forall s,t\in B$, if $s\leq t$ then $s^{-1}\leq t^{-1}$. \label{ESNcondition2}
	\item $\forall s,t,y,z\in B$, if $s\leq y$, $t\leq z$, $s t\neq 0$, and $y z \neq 0$, then $s t \leq y z$. \label{ESNcondition3}
	\item $\forall e,s\in B$, if $e\leq \dom(s)$, then $\exists!t\leq s$ such that $\dom(t)=e$. \label{ESNcondition4}
	\item $\forall e,s\in B$, if $e\leq \ran(s)$, then $\exists!t\leq s$ such that $\ran(t)=e$. \label{ESNfinalcondition}
\end{enumerate}
For $b\in B$, let
\[
\overline b = \sum_{a\in B: a\leq b} a.
\]
Then $\overline{B} = \{\overline b :b\in B\}$ is an inverse semigroup under matrix multiplication, with 
\begin{equation}
\label{eqSize}
|\overline{B}| = \sum_{i=1}^k r_i^2 |G_i|.
\end{equation}
Furthermore $E(\overline B)=\{\bar e : e\in E(B)\}$ and the map $E(B) \rightarrow E(\overline B)$ given by $e\mapsto \bar e$ is an isomorphism of posets. If $B_i$ denotes the natural basis of $M_{r_i}(\C G_i)$, then the $\D$-classes of $\overline B$ are $D_1 ,\ldots, D_k$, with $D_i = \{\bar b: b\in B_i\}$. If $e\in E(B)\cap B_i$, we have $G_{\bar e} \cong G_i$. Finally, the natural partial order $\leq$ on $\overline B$ is given by $\overline s\leq \overline t \iff s\leq t$.
\end{thm}


The ESN theorem also asserts that the construction in Theorem \ref{ThmESN} is sufficient to construct any finite inverse semigroup $S$ up to isomorphism. In particular, we have the following.

\begin{thm}[ESN theorem pt.\ 2]
\label{ThmESN2}
Suppose $S$ is a finite inverse semigroup and the partition of $E(S)$ obtained by restricting Green's $\D$-relation on $S$ to $E(S)$ is $\{X_1 ,\ldots, X_k\}$. Suppose $|X_i|=r_i$ and
$G_{X_i} \cong G_i$ for all $i\in \{1 ,\ldots, k\}$. Let $\sqsubseteq$ be a partial order on $E(B)$ for which $(E(S),\leq) \cong (E(B),\sqsubseteq)$. Then $S$ may be obtained, up to isomorphism, from the construction of Theorem \ref{ThmESN} for some partial order on $B$ which restricts to $\sqsubseteq$ on $E(B)$. 
\end{thm}

Theorem \ref{ThmESN2} is constructive. Our statement of Theorem \ref{ThmESN2} is perhaps nonstandard. We include a proof in Section \ref{SecESN2Pf}.

%

\subsection{Our enumeration algorithms}
\label{SecAlgms}

We take $\N=\{1, 2, \ldots\}$.
Recall that a {\em composition} of $m\in \N$ is a list of positive integers whose sum is $m$. An {\em integer partition} (or just {\em partition}) of $m\in \N$ is a composition of $m$ whose elements are in weakly decreasing order. If $\lambda$ is a partition or composition, denote the number of entries of $\lambda$ (also called the {\em length} of $\lambda$) by $|\lambda|$ and denote the $i$th entry of $\lambda$ by $\lambda_i$. 
Any set partition $P=\{X_1 ,\ldots, X_k\}$ of a finite set $X$ gives rise to a partition $\lambda$ of $|X|$ called the {\em shape} of $P$. Specifically, by relabeling if necessary we may assume $|X_1|\geq \cdots \geq |X_k|$; then the shape of $P$ is the partition $(|X_1| ,\ldots, |X_k|)$.

\begin{defn} If $m,n\in \N$, $\lambda$ is a partition of $m$, and $C$ is a composition with $|C|=|\lambda|$, we say $C$ is an {\em admissible composition} for $(n,\lambda)$ if 
\[
n=\sum_{i=1}^{|\lambda|}\lambda_i^2 C_i.
\]
\end{defn}
For instance, the admissible compositions for $(10,(2,1))$ are $(2,2)$ and $(1,6)$, and the only admissible composition for $(n,(\underbrace{1 ,\ldots, 1}_n))$ is $(\underbrace{1 ,\ldots, 1}_n)$.

%
%

We shall write the pseudocode for our algorithms using \texttt{Python}-esque syntax. In particular $[\up{ }]$ denotes a new empty list, $\tt{dict}([ \up{ }])$ denotes a new empty dictionary, $[x]$ denotes a new list containing $x$, and if $x$ is a key in a dictionary $d$, then $d[x]$ denotes the associated value. A single equals sign denotes variable assignment, and nonempty lists are indexed beginning at $0$. We now give the specifications for the functions we shall use in our enumeration algorithms. 

\begin{itemize}

	\item \texttt{Groups($n$)} accepts $n\in \N$, chooses a representative from each isomorphism class of the groups of order not exceeding $n$, and returns this set of representatives. This function can be computed easily with standard computational mathematical suites such as \texttt{Sage} \cite{sage} or \texttt{GAP} \cite{GAP4}.

	\item \texttt{Partitions($m$)} accepts $m\in \N$ and returns the set of partitions of $m$. 
	\item \texttt{AdmissibleCompositions($n$,$\lambda$)} accepts $n\in \N$ and a partition $\lambda$ of a positive integer not exceeding $n$, and returns the set of admissible compositions for $(n,\lambda)$.
	\item \texttt{MeetSemilattices($m$)} accepts $m\in \N$ and returns the set of meet-semilattices of order $m$ up to isomorphism. The algorithm of Heitzig and Reinhold for enumerating finite lattices up to isomorphism \cite{CountingFiniteLattices} may be used to implement this function by returning the lattices of order $m+1$ up to isomorphism with their maximal elements removed.
	\item \texttt{DPartitions($E$,$\lambda$)} accepts a finite meet-semilattice $E$ and a partition $\lambda$ of $|E|$, and returns the $\D$-partitions of $E$ of shape $\lambda$ as ordered tuples. Specifically, if $P=\{X_1 ,\ldots, X_k\}$ is a $\D$-partition of $E$ of shape $\lambda$, then by relabeling if necessary we may assume $|X_1|\geq \cdots \geq |X_k|$. This function outputs $(X_1 ,\ldots, X_k)$ for $P$.
	\item \texttt{GroupMaps($P$,$C$,$L$)} accepts a composition $C$, a tuple $P=(X_1 ,\ldots, X_{|C|})$ of length $|C|$, and a set $L$ of groups, and returns the set of all functions $f: \{X_1 ,\ldots, X_{|C|}\} \rightarrow L$ such that $|f(X_i)|=C_i$ for all $i\in \{1 ,\ldots, |C|\}$.
	\item Write $()$ for the identity of any group. \texttt{EGroupoid($E$,$P$,$f$)} accepts a finite meet-semilattice $E$, a $\D$-partition $P=(X_1 ,\ldots, X_k)$ of $E$ (with $|X_1|\geq \cdots \geq |X_k|$), and a function $f$ for which $f(X_i)$ is a finite group for all $i\in \{1 ,\ldots, k\}$, and returns $(B,\leq_{E(B)})$, where $B$ is the natural basis of the algebra $\bigoplus_{i=1}^k M_{|X_i|}(\C f(X_i))$ (where each block is indexed by the elements of $X_i$), and $\leq_{E(B)}$ is the partial order on $E(B)$ given by, for $a,b\in E$,
	\[
	()_{a,a} \leq_{E(B)} ()_{b,b} \iff a\leq b.
	\]
	\item \texttt{GPosets($G$)} accepts an output $(B,\leq_{E(B)})$ of \texttt{EGroupoid}, and returns the set of partial orders $\leq$ on $B$ which restrict to $\leq_{E(B)}$ on $E(B)$ and meet the hypotheses of Theorem \ref{ThmESN}. An implementation of this function is given in Section \ref{SecPartialOrders}.
	\item \texttt{ESN}($G$,$\leq$) accepts an output $G$ of \texttt{EGroupoid} an output $\leq$ of \texttt{GPosets}, and returns the inverse semigroup $S$ obtained from the construction of Theorem \ref{ThmESN} applied to $(G,\leq)$, after renaming $\overline{()_{e,e}}$ as $e$ for all $e\in E$. 
	\item \texttt{Invariants}($S$) accepts a finite inverse semigroup $S$ and returns a tuple $I(S)$ having the property that, for inverse semigroups $S$, $T$, if $E(S)=E(T)$ and $S\cong T$, then $I(S)=I(T)$. An implementation of this function is given in Section \ref{SecInvariants}.
	\item \texttt{IsNew}($S$,$I$,\textup{isgs}) accepts a finite inverse semigroup $S$, $I$=\texttt{Invariants}($S$), and a dictionary \textup{isgs}. If the key $I$ is not present in \textup{isgs}, this function returns True. On the other hand, if the key $I$ is present in \textup{isgs}, then this function returns False if $S$ is isomorphic to some inverse semigroup in the list \textup{isgs}[$I$], and returns True otherwise. An implementation of this function, incorporating our implementation of \texttt{Invariants}, is given in Section \ref{SecIsNew}.
	\item \texttt{Output}(\textup{isgs}) accepts a dictionary \textup{isgs} and outputs to file (or console) every element of \textup{isgs}[$I$], for every key $I$ in \textup{isgs}.
\end{itemize}

\begin{alg}\label{AlgMain}Algorithm for enumerating the inverse semigroups of order $n$ up to isomorphism.
\lstset{ %
basicstyle=\footnotesize,       
numbers=left,                   
numberstyle=\footnotesize,      
stepnumber=1,                   
numbersep=5pt,                  
backgroundcolor=\color{white},  
showspaces=false,               
showstringspaces=false,         
showtabs=false,                 
frame=single,           
tabsize=2,          
captionpos=b,           
breaklines=true,        
breakatwhitespace=false,    
escapeinside={\*}{*}         
}

\begin{lstlisting}
Input:*$n$*

*$\textup{Gn}=\texttt{Groups($n$)}$ \label{MakeGpsStep}*

for *$m$* in *\{1 ,\ldots, n\}*: *\label{OneThruNStep}*
	for *$\lambda$* in *$\texttt{Partitions}(m)$*:
		*$S_\lambda=\texttt{AdmissibleCompositions}(n,\lambda)$*

	for *$E$* in *$\texttt{MeetSemilattices}(m)$*: *\label{MainStepE}*
		for *$\lambda$* in *$\texttt{Partitions}(m)$*: *\label{MainStepLambdaInPartitions}*
			*$\textup{isgs}=\tt{dict}([\up{ }])$* *\label{NewDictStep}*
			for *$C$* in *$S_\lambda$*: *\label{AdmissIterStep}*
				for *$P$* in *$\texttt{DPartitions}(E,\lambda)$*: *\label{DPartitionsStep}*
					for *$f$* in *$\texttt{GroupMaps}(P,C,\textup{Gn})$*:
						*$G=\texttt{EGroupoid}(E,P,f)$*
						for *$\leq$* in *$\texttt{GPosets}(G)$*:
							*$S=\texttt{ESN}(G,\leq)$*
							*$I$*=*$\texttt{Invariants}(S)$* *\label{IsomFirstStep}*
							if *$\texttt{IsNew}(S,I,\textup{isgs})$*:
								if *$I$* *$\textup{in}$* *$\textup{isgs}$*:
									*$\textup{isgs}[I]$*.*$\textup{append}(S)$*
							  else:
							  	*$\textup{isgs}[I]=[S]$* *\label{IsomLastStep}*
			*$\texttt{Output}(\textup{isgs})$*
	
\end{lstlisting}
\end{alg}

We note that Algorithm \ref{AlgMain} is easily parallelized at line \ref{MainStepE} by starting a new task for each meet-semilattice $E$. We also note that the renaming in the specification of $\texttt{ESN}$ guarantees that if $S$ and $T$ are inverse semigroups generated by this algorithm for the same meet-semilattice $E$ (that is, within the same pass of the loop beginning at line \ref{MainStepE}), then $E(S)=E(T)=E$.

\begin{thm}[Correctness of Algorithm \ref{AlgMain}]
The output of Algorithm \ref{AlgMain} is precisely the collection of inverse semigroups of order $n$ up to isomorphism.
\end{thm}
\begin{proof}
Let $\G_n=\texttt{Groups}(n)$. By Theorem \ref{ThmESN2}, we may obtain the inverse semigroups of order $n$ up to isomorphism by iterating over all possible combinations of meet-semilattices $E$ of order $1 ,\ldots, n$ up to isomorphism, set partitions $P=\{X_1 ,\ldots, X_k\}$ of $E$ (for $k \in \{1,\ldots,|E|\}$), functions $f:P\rightarrow \G_n$, and partial orders $\leq$ on the natural basis $B$ of $A=\bigoplus_{i=1}^{k} M_{|X_i|}(\C f(X_i))$ (where the rows and columns of the $X_i$ block of $A$ are indexed by $X_i$) for which
\begin{itemize}
	\item $a\leq b \in E \iff ()_{a,a}\leq ()_{b,b} \in E(B)$,
	\item the hypotheses of Theorem \ref{ThmESN} are satisfied, and
	\item $n=\sum_{i=1}^{k} |X_i|^2 |f(X_i)|$,
\end{itemize}
and outputting each inverse semigroup $S$ afforded by the construction of Theorem \ref{ThmESN}, provided $S$ is not isomorphic to any previously-output inverse semigroup. We show that Algorithm \ref{AlgMain} is an implementation of this procedure. In particular, we must justify lines \ref{AdmissIterStep}, \ref{DPartitionsStep}, \ref{IsomFirstStep}--\ref{IsomLastStep}, and the placement of line \ref{NewDictStep} of Algorithm \ref{AlgMain}.

If $S$ and $T$ are isomorphic inverse semigroups then it is straightforward to verify that $E(S)\cong E(T)$, both as posets and as inverse semigroups under the meet operation. Therefore we only need to test newly-generated inverse semigroups $S$ against previously-generated inverse semigroups $T$ for which $E(S)=E(T)$. In particular, line \ref{NewDictStep} may be placed below line \ref{MainStepE}. Furthermore, if $S$ and $T$ are isomorphic finite inverse semigroups with $E(S)=E(T)=E$, and $P_1$ and $P_2$ are the set partitions of $E$ induced by the restrictions of $\D$ on $S$ and $T$, respectively, to $E$, then by parts \eqref{Isom3} and \eqref{Isom5} of Theorem \ref{ThmIsomTest}, $P_1$ and $P_2$ have the same shape.
We therefore also only need to test newly-generated inverse semigroups $S$ for isomorphism against previously-generated inverse semigroups $T$ for which the shape of $\D$ restricted to $E(S)$ is equal to the shape of $\D$ restricted to $E(T)$. In particular, the placement of line \ref{NewDictStep} is correct.

By Theorem \ref{ThmLegitPartition} we only need to consider $\D$-partitions of each meet-semilattice $E$, so line \ref{DPartitionsStep} is correct.

Let $S$ be a semigroup generated by the procedure in the first paragraph of the proof, generated from the parameters $P$ and $f$. By relabeling if necessary, suppose $|X_1| \geq \cdots \geq |X_k|$. Then the composition $(|f(X_1)|, \ldots, |f(X_i)|)$ is an admissible composition for $(n,(|X_1| ,\ldots, |X_k|))$. Therefore line \ref{AdmissIterStep} of the algorithm is correct. 

Lines \ref{IsomFirstStep}--\ref{IsomLastStep} of the algorithm sort the generated inverse semigroups according to their invariants for isomorphism testing, and therefore serve only to make the algorithm more efficient. In particular, since the placement of line \ref{NewDictStep} is correct, in lines \ref{IsomFirstStep}--\ref{IsomLastStep} any newly-generated inverse semigroup $S$ is tested for isomorphism against every previously-generated inverse semigroup $T$ to which $S$ could possibly be isomorphic. 
\end{proof}

\begin{rmk}
If one desires to enumerate only inverse monoids, a simple modification of Algorithm \ref{AlgMain} can be used to do so. To enumerate the inverse monoids of order $n$ instead of the inverse semigroups of order $n$, iterate over the lattices of order $1 ,\ldots, n$ instead of the meet-semilattices of order $1 ,\ldots, n$ on line \ref{MainStepE}. (A finite inverse semigroup $S$ is a monoid if and only if $E(S)$ is a lattice.)
\end{rmk}

\begin{rmk}
We comment on the idea behind the use of \texttt{AdmissibleCompositions} and line \ref{AdmissIterStep} in Algorithm \ref{AlgMain}. Given a finite meet-semilattice $E$, to iterate over the $\D$-partitions of $E$ it would suffice to iterate over all set partitions of $E$ and check which ones are $\D$-partitions. However for the purpose of generating the inverse semigroups of order $n$ this is highly inefficient, as the number of set partitions of $E$ grows rapidly as $|E|$ grows, and for a typical meet-semilattice $E$ of order $m\leq n$, the vast majority of set partitions $P$ of $E$ cannot serve to generate an inverse semigroup of order $n$ simply because of \eqref{eqSize}. In particular, if $P=\{X_1 ,\ldots, X_k\}$ is a $\D$ partition of $E$ for which there exists an inverse semigroup of order $n$ such that $E(S)=E$ and the restriction of $\D$ on $S$ to $E$ is $P$, then
\[
|S| = n = \sum_{i=1}^k |X_i|^2 \lambda_i
\]
for some $\lambda_1 ,\ldots, \lambda_k \in \N$ (namely, $\lambda_i = |G|$ for any group $G$ such that $G_{X_i}\cong G$). Indeed, if $m\in \{n,n-1\}$, then the {\em only} partition $P$ of $E$ we need to consider is the finest partition of $E$ (which is automatically a $\D$-partition). If $m\in \{n-2,n-3\}$ then we only need to consider partitions of $E$ of shape $(1 ,\ldots, 1)$ and $(2, 1 ,\ldots, 1)$, and so on. In general, since we only need to consider set partitions of $E$ whose shapes $\lambda$ have an admissible composition for $(n,\lambda)$, this reduces the amount of work (done by $\texttt{DPartitions}$) required to generate all possible $\D$-partitions of $E$ needed by the rest of the algorithm.
\end{rmk}

Now, let $E$ be a finite meet-semilattice, $P$ a $\D$-partition of $E$, and $f$ a function from $P$ to $\texttt{Groups}(n)$. Algorithm \ref{AlgMain} is easily modified to enumerate the inverse semigroups $S$ for which $E(S)=E$, $\D$ restricted to $E$ is equal to $P$, and $\forall e\in E$ $\forall X\in P$, $e\in X \implies G_e \cong f(X)$. In particular, we have the following.


\begin{alg} Algorithm for enumerating the inverse semigroups having a specified semilattice of idempotents $E$, restriction of $\D$ to $E$, and collection of maximal subgroups.
\label{AlgBasic}
\lstset{ %
basicstyle=\footnotesize,       
numbers=left,                   
numberstyle=\footnotesize,      
stepnumber=1,                   
numbersep=5pt,                  
backgroundcolor=\color{white},  
showspaces=false,               
showstringspaces=false,         
showtabs=false,                 
frame=single,           
tabsize=2,          
captionpos=b,           
breaklines=true,        
breakatwhitespace=false,    
escapeinside={\*}{*}         
}
\begin{lstlisting}
Input:*$E, P, f$*

*$\textup{isgs}=\tt{dict}([\up{ }])$*
*$G=\texttt{EGroupoid}(E,P,f)$*
for *$\leq$* in *$\texttt{GPosets}(G)$*:
	*$S=\texttt{ESN}(G,\leq)$*
	*$I$*=*$\texttt{Invariants}(S)$*
	if *$\texttt{IsNew}(S,I,\textup{isgs})$*:
		if *$I$* *$\textup{in}$* *$\textup{isgs}$*:
			*$\textup{isgs}[I]$*.*$\textup{append}(S)$*
	  else:
	  	*$\textup{isgs}[I]:=[S]$*
*$\texttt{Output}(\up{isgs})$*
	
\end{lstlisting}
\end{alg}

Note that Algorithm \ref{AlgBasic} is not strictly a subroutine of Algorithm \ref{AlgMain}, as in Algorithm \ref{AlgMain} a wider amount of isomorphism testing is necessary. In particular, in Algorithm \ref{AlgMain} any newly generated inverse semigroup must be tested for isomorphism against any previously generated inverse semigroup having the same underlying partition for the restriction of $\D$ to its semilattice of idempotents.



%
%


\section{\texttt{GPosets}}

\label{SecPartialOrders}

In this section we give an implementation of the \texttt{GPosets} function for Algorithms  \ref{AlgMain} and \ref{AlgBasic}. 

\begin{defn}
\label{defLevel}
Let $(Y,\leq)$ be a finite poset. The {\em down-levels} (or just {\em levels}) of $Y$ are defined inductively. Let $Y_1=Y$. For $i\in \N$, let $L_i$ consist of the maximal elements of $Y_i$, and let $Y_{i+1}$ be the poset obtained by removing $L_i$ from $Y_i$. Let $d$ be the first value of $i$ for which $L_{d+1}=\varnothing$. We call $\{L_1 ,\ldots, L_d\}$ the set of {\em down-levels} of $Y$. For $i\in \{1 ,\ldots, d\}$, the set $L_i$ is called the {\em $i$th down-level of $Y$}.  
\end{defn}

By definition, the levels of $Y$ form a partition of $Y$.

Let $(E,\leq)$ be a finite meet-semilattice and $P=\{X_1 ,\ldots, X_k\}$ a $\D$-partition of $E$. Let $G_1 ,\ldots, G_k$ be finite groups. For $i\in \{1 ,\ldots, k\}$, let us denote the natural basis of $M_{|X_i|}(\C G_i)$ by $B_i$. Let $B=\cup_{i=1}^k B_i$. Write $()$ for the identity of any group. Let $\leq_{E(B)}$ be the partial order on $B$ inherited from $(E,\leq)$---for $a,b\in E$,
\[
()_{a,a} \leq_{E(B)} ()_{b,b} \iff a\leq b.
\]

%
%

Denote the levels of $E$ be $L_1 ,\ldots, L_d$. If there exists an idempotent $()_{r,r}\in B_i$ such that $r\in L_j$, indicate this by writing $B_i \cap L_j \neq \varnothing$. It is possible to have $B_i \cap L_{j_1} \neq \varnothing$ and $B_i \cap L_{j_2} \neq \varnothing$ for distinct $L_{j_1},L_{j_2}$. Let $I(B_i) = \max \{j: B_i \cap L_j \neq \varnothing\}$. Define a linear order $\prec$ on the $B_i$ by setting 
\begin{equation}
\label{EqBiOrder}
\{B_i : I(B_i) = d\} \prec \{B_i: I(B_i) = d-1\} \prec \cdots \prec \{B_i:I(B_i)=1\},
\end{equation}
and then ordering the $B_i$ within each set in \eqref{EqBiOrder} arbitrarily. By relabeling if necessary, we may assume that $B_1\prec B_2 \prec \cdots \prec B_k$. If there exist idempotents $f\in B_i$ and $e\in B_j$ such that $f$ covers $e$, let us say that $B_i$ covers $B_j$.

With this ordering of the $B_i$ we will build the partial orders on $B$ that we seek by finding all extensions of the partial order $\leq_{E(B)}$ on $E(B)$ to $E(B) \cup B_1$, and then finding all extensions of such extensions to $E(B) \cup B_1 \cup B_2,$ and so on. Let $$\widehat B_i = E(B) \cup B_1 \cup \cdots \cup B_i.$$

Our implementation of $\texttt{GPosets}((B,\leq_{E(B)}))$ may be described as a depth-first search in the following search tree. 
Let $\leq_1 = \leq_{E(B)}$.
The root of our search tree is $(\widehat B_1,\leq_1,1)$. Nodes of our search tree are of the form $(\widehat B_i, \leq_{i}, i)$ (where $\leq_i$ is a partial order on $\widehat B_i$), for $i\in \{1 ,\ldots, k\}$, and for a node $N=(\widehat B_i, \leq_{i}, i)$ with $i<k$, the children of $N$ are given by the following algorithm, with sub-functions as specified following the algorithm.

\begin{alg} $\texttt{Children}((\widehat B_i,\leq_i,i))$
\label{AlgGPosetsChildren}
\lstset{ %
basicstyle=\footnotesize,       
numbers=left,                   
numberstyle=\footnotesize,      
stepnumber=1,                   
numbersep=5pt,                  
backgroundcolor=\color{white},  
showspaces=false,               
showstringspaces=false,         
showtabs=false,                 
frame=single,           
tabsize=2,          
captionpos=b,           
breaklines=true,        
breakatwhitespace=false,    
escapeinside={\*}{*}         
}
\begin{lstlisting}
Input:*$\widehat B_i, \leq_{i}, i$*

*$L=[\textup{ }]$*
for each *$B_{j}$* covered by *$B_{i+1}$*:
	*$\widetilde \leq = \texttt{PartialOrderRestriction}(\leq_{E(B)},{B_{i+1}\cup B_j})$*
	*$L$*.append(*$\texttt{PosetPossibilities}(B_{i+1},B_j,\widetilde \leq)$*)
for *$R$* in *$\texttt{CartesianProduct}(L)$*:
	*$\leq_{i+1} = \texttt{TransitiveClosure}(\leq_i,R)$*
	if *$\texttt{PassesCardinalityTest}(\widehat B_{i+1},\leq_{i+1},i)$*:
		yield *$(\widehat B_{i+1},\leq_{i+1},i+1)$*

	
\end{lstlisting}
\end{alg}

\begin{itemize}
\item $\texttt{PartialOrderRestriction}(\leq_{E(B)},{B_{i+1}\cup B_j})$ returns the restriction of the partial order $\leq_{E(B)}$ to $E(B)\cap \left(B_{i+1}\cup B_j\right)$.
	\item $\texttt{PosetPossibilities}(B_{i+1},B_j,\widetilde \leq)$ returns the set of partial orders $\leq$ on $B_{i+1} \cup B_j$ which restrict to the partial order $\widetilde \leq$ on $E(B)\cap \left(B_{i+1}\cup B_j \right)$, which restrict to equality on $B_j$ and on $B_{i+1}$, and for which
\begin{enumerate}
	\item[(R1)] \label{R1} $\forall s,t\in B_j \cup B_{i+1}$, if $s\leq t$, then $s^{-1}\leq t^{-1}$,
	\item[(R2)] \label{R2} $\forall s,t,y,z \in B_j \cup B_{i+1}$, if $s\leq y$, $t\leq z$, $s t\neq 0$, and $yz\neq 0$, then $s t\leq yz$,
	\item[(R3)] \label{R3} $\forall e,s\in B_j \cup B_{i+1}$, if $e\leq \dom(s)$, then $\exists! t\in B_j \cup B_{i+1}$ with $t \leq s$ and $\dom(t)=e$, and
	\item[(R4)] \label{R4} $\forall e,s\in B_j \cup B_{i+1}$, if $e\leq \ran(s)$, then $\exists! t\in B_j \cup B_{i+1}$ with $t\leq s$ and $\ran(t)=e$.
\end{enumerate}
This function can be computed in a straightforward depth-first (or even brute-force) manner. For the purposes of Algorithm \ref{AlgMain}, up to isomorphism the number of possible combinations of inputs to this function is very small relative to the number of times this function will be called, and the performance of the algorithm is improved considerably by computing the output of this function just once per combination of inputs, caching the results and looking them up as necessary.
	\item $\texttt{CartesianProduct}(L)$ takes a list $L=[S_1 ,\ldots, S_j]$ of sets or lists, and returns the set $S_1 \times \cdots \times S_j$.
	\item $\texttt{TransitiveClosure}(\leq_i,R)$ accepts a partial order $\leq_i$ and a tuple $R$ of partial orders, and returns the partial order given by the transitive closure of $ \leq_i \cup \left(\cup_{r\in R} r \right).$
	\item Fix any idempotent $e\in B_{i+1}$ and, for $i\in \{1,\ldots,i\}$, let $r(h)=|\{f\in B_h:f\leq_{i+1} e\}|$. $\texttt{PassesCardinalityTest}(\widehat B_{i+1},\leq_{i+1},i)$ returns True if $|\{t\in B_h:t\leq_{i+1} s\}| = r(h)$ for all $s\in B_{i+1}$ and all $h\in \{1 ,\ldots, i\}$, and returns False otherwise.
\end{itemize}

The output of $\texttt{GPosets}((B,\leq_{E(B)}))$ is the collection of partial orders $\leq_k$ in the leaves of this search tree. The proof of correctness of this implementation of $\texttt{GPosets}$ is somewhat lengthy, and is included in Section \ref{SecGPosetsCorrect}.


\section{\texttt{Invariants} and \texttt{IsNew}}

In this section we discuss our implementation of the \texttt{Invariants} and \texttt{IsNew} functions for Algorithms \ref{AlgMain} and \ref{AlgBasic}. Our implementation of \texttt{IsNew} makes use of our implementation of \texttt{Invariants}, which we discuss first.

\label{SecIsom}


\subsection{\texttt{Invariants}}

\label{SecInvariants}

For the purposes of our algorithms, a {\em tuple of invariants} for an inverse semigroup S is a tuple $I(S)$ having the property that for finite inverse semigroups $S,T$, if $E(S)=E(T)$ and $S\cong T$, then $I(S)=I(T)$. In Algorithms \ref{AlgMain} and \ref{AlgBasic} we use invariants to sort generated inverse semigroups for isomorphism testing---when a new inverse semigroup $S$ is generated, it must be tested for isomorphism only against previously-generated inverse semigroups $T$ for which $I(S)=I(T)$. In this section we give the invariants we used in our implementation of Algorithm \ref{AlgMain}. 

\begin{rmk}The invariants specified in this section are remarkably efficient at separating inverse semigroups in Algorithm \ref{AlgMain}. It is not uncommon for Algorithm \ref{AlgMain} to be able to accept a newly generated inverse semigroup $S$ with no isomorphism testing whatsoever. We report on the frequency of this occurrence for inverse semigroups of order $\leq 15$ in Section \ref{SecResults}. However, for Algorithm \ref{AlgBasic} (in which $E$, $P$, and $f$ are specified and are therefore identical for every generated inverse semigroup), the invariants given in this section are identical for every generated inverse semigroup. We have included the use of the function \texttt{Invariants} in Algorithm \ref{AlgBasic} as it is conceptually correct and would be useful in an implementation of Algorithm \ref{AlgBasic} if a more refined tuple of invariants could be given. 
\end{rmk}

%
It is straightforward to verify the following theorem. 

\begin{thm}
\label{ThmPosetIsom2}
Suppose $X$ and $Y$ are isomorphic finite posets, with $\phi:X\rightarrow Y$ an isomorphism. If the levels of $X$ are $X_1 ,\ldots, X_{m}$ and the levels of $Y$ are $Y_1 ,\ldots, Y_{m'}$, then $m=m'$ and $\phi$ restricts to a bijection between $X_i$ and $Y_i$ for all $i$.
\end{thm}

Let $S$ and $T$ be finite inverse semigroups. Recall that by part \eqref{ThmPosetIsom1} of Theorem \ref{ThmIsomTest} we have that if $S\cong T$, then as posets $(S,\leq)\cong (T,\leq)$.
Let $S_1 ,\ldots, S_{m}$ denote the levels of $(S,\leq)$ and $T_1 ,\ldots, T_{m'}$ denote the levels of $(T,\leq)$. 
By Theorem \ref{ThmPosetIsom2}, then, if $S\cong T$ then $\Lev(S) = \Lev(T)$.


\begin{defn} Let $(X,\leq)$ be a finite poset. The {\em up-levels} of $X$ are defined by replacing {\em maximal} with {\em minimal} in Definition \ref{defLevel}. 
\end{defn}

\begin{defn}Let $E$ be a finite meet-semilattice and let $L^U$ and $L^D$ denote the up-levels and the down-levels of $E$, respectively. The {\em up-down} levels of $E$ is the meet $L^U\meet L^D$ in the lattice of set partitions of $E$. That is, the up-down levels of $E$ is given by the collection of nonempty pairwise intersections between the elements of $L^U$ and $L^D$.
\end{defn}

Let $S$ be a finite inverse semigroup with $E=E(S)$. Let $D_{S,E}$ be the restriction of Green's $\D$-relation to $E(S)$, so $D_{S,E}$ is a $\D$-partition of $E$. For $e\in E$, let $D_{S,E}(e)$ denote the element of $D_{S,E}$ containing $e$, and let $G_S(e)$ be the isomorphism class of the maximal subgroup of $S$ at $e$. For each element $L$ of the up-down layers of $E$, let $X_S(L)$ denote the multiset of ordered pairs $X_S(L)=\{(|D_{S,E}(e)|,G_S(e)) : e\in L\}$. 

\begin{thm}
\label{ThmTestAgainst}Let $S$ and $T$ be finite inverse semigroups with $E(S)=E(T)=E$. Let $U$ denote the up-down levels of $E$. If $S\cong T$, then $X_S(L)$=$X_T(L)$ for all $L\in U$.
\end{thm}

\begin{proof}
Suppose $\phi:S\rightarrow T$ is an isomorphism. Then $\phi$ restricted to $E(S)=E$ is an automorphism of posets $\phi|_{E} : E \rightarrow E(T)=E$. Any automorphism of a finite poset must preserve the up-levels and the down-levels of that poset, and hence must preserve the up-down levels of that poset as well. Thus, $\phi$ preserves the up-down levels of $E$. Note that $|D_{S,E}(e)|$ is simply the number of idempotents in the $\D$-class of $e$. Since $\phi$ also maps $\D$-classes to $\D$-classes and restricts to isomorphisms of maximal subgroups, if $L\in U$ and $e\in L$, we have $\phi(e)\in L$, $G_S(e) = G_S({\phi(e)})$, and $|D_{S,E}(e)| = |D_T(\phi(e))|$. Thus, as multisets we have $X_S(L) = X_T(L)$. 
\end{proof}

For the purposes of Algorithm \ref{AlgMain}, for an inverse semigroup $S$ generated by the algorithm we let $\texttt{Invariants}(S) = (\Lev(S),X_S)$, where $X_S$ is the function $L\mapsto X_S(L)$. It follows that if $S$ and $T$ are inverse semigroups generated by Algorithm \ref{AlgMain} and $S\cong T$ then $\tt{Invariants}(S) = \tt{Invariants}(T)$. 


\subsection{\texttt{IsNew}} 
\label{SecIsNew}
In this section all inverse semigroups are assumed to have been created by the function $\texttt{ESN}$ in Algorithm \ref{AlgMain} or \ref{AlgBasic}. In particular, non-idempotent elements of inverse semigroups in this section are of the form $\overline{g_{a,b}}$ for some group element $g$ and some idempotents $a,b\in E(S)$. To simplify the notation in this section we write the idempotents of $S$ in the same way, so we write $\overline{()_{e,e}}$ for $e\in E(S)$. Further, let $\G_n$ be the set constructed by $\texttt{Groups}(n)$ on line \ref{MakeGpsStep} of Algorithm \ref{AlgMain}. For $e\in E$ we have $G_e=\{\overline{g_{e,e}}:g\in G\}$ for some $G\in \G_n$, so there is an obvious isomorphism between $G_e$ and $G$.

The function $\texttt{IsNew}(S,I,\textup{isgs})$ accepts an inverse semigroup $S$, $I=\texttt{Invariants}(S)$, and a dictionary \textup{isgs} such that $\textup{isgs}[I]$ is list of inverse semigroups $T$ for which $\texttt{Invariants}(S)=\texttt{Invariants}(T)=I$, $E(S)=E(T)=E$, and the shape of Green's $\D$-relations on $S$ and $T$, restricted to $E$, are equal. The function $\texttt{IsNew}(S,I,\textup{isgs})$ returns False if $S\cong T$ for some $T\in \textup{isgs}[I]$ and returns True otherwise. We implement this by iterating over the inverse semigroups $T\in \textup{isgs}[I]$ and for each such $T$, determining whether or not $S$ and $T$ are isomorphic.

Let \texttt{Invariants} be as implemented in the Section \ref{SecInvariants} and let $S$ and $T$ be finite inverse semigroups with $\texttt{Invariants}(S)=\texttt{Invariants}(T)=I$ and $E(S)=E(T)=E$. Since $\Lev(S)=\Lev(T)$ we have $|S|=|T|$, and it is further straightforward to prove that $S$ and $T$ have the same number of lonely idempotents. Let $\texttt{IsIsoc}(S,T)$ be True if $S \cong T$ and False otherwise. We now describe our implementation of $\texttt{IsIsoc}$. Our strategy for computing $\texttt{IsIsoc}(S,T)$ is to find a homomorphism from $S$ to $T$ among the bijections from $S$ to $T$ or certify that among these bijections no homomorphism exists. Fortunately, frequently it is only necessary to check a small number of these bijections to find an isomorphism or certify non-isomorphism.

Let $D_{S,E}$ be the partition of $E=E(S)$ obtained by restricting Green's $\D$-relation on $S$ to $E$. We specify the following functions.

\begin{itemize}
	\item $\tt{DRestriction}(S)$ takes $S$ and returns the partition $D_{S,E}$. The data structure for $D_{S,E}$ must be implemented in such a way that the iteration order of $D_{S,E}$ is fixed.
	\item $\tt{EColoring}(S,D_{S,E})$ returns a copy of $E=E(S)$, with nodes colored in the following manner. Let $e\in E(S)$ and fix an ordering $Sl_1 ,\ldots, Sl_j$ of the lonely idempotents of $S$.
	\begin{itemize}
		\item If $\overline{()_{e,e}} \in S$ is not a lonely idempotent, $e$ is colored by $(G,|D_{S,E}(e)|)$, where $G\in \G_n$ is such that $G_e \cong G$.
		\item If $\overline{()_{e,e}} \in S$ is a lonely idempotent, say $Sl_i$, then $e$ is colored by $i$.
	\end{itemize}
	\item $\tt{ColoredIsoms}(E_S,E_T)$ accepts two outputs $E_S$ and $E_T$ of $\texttt{EColoring}$ and returns the set of color-preserving (poset) isomorphisms from $E_S$ to $E_T$. This function can be computed with standard graph-theoretic software such as $\tt{nauty}$ \cite{Nauty2}.
	\item $\texttt{MaxSubgp}(X)$ takes an element $X\in D_{S,E}$ and returns the element $G$ of $\G_n$ for which $G_e\cong G$ for any $e\in X$.
	\item $\texttt{Bijections}(A,B)$ takes two equal-sized sets or lists $A$ and $B$, and returns the set of bijections from $A$ to $B$.
	\item $\texttt{IsISGHomomorphism}(d)$ takes a dictionary $d$ whose keys are elements of an inverse semigroup $S$ and whose values are elements of an inverse semigroup $T$ (that is, $d$ is a map from $S$ to $T$), and returns True if $d$ is a homomorphism and False otherwise.
\end{itemize}

\begin{alg} Implementation of $\texttt{IsIsoc}(S,T)$
\label{AlgIsIsoc}
\lstset{ %
basicstyle=\footnotesize,       
numbers=left,                   
numberstyle=\footnotesize,      
stepnumber=1,                   
numbersep=5pt,                  
backgroundcolor=\color{white},  
showspaces=false,               
showstringspaces=false,         
showtabs=false,                 
frame=single,           
tabsize=2,          
captionpos=b,           
breaklines=true,        
breakatwhitespace=false,    
escapeinside={\*}{*}         
}
\begin{lstlisting}
Input:*$S, T$*

*$D_{S,E} = \texttt{DRestriction}(S)$* *\label{LineSetup1}*
*$D_{T,E} = \texttt{DRestriction}(T)$*
*$E_S=\tt{EColoring}(S,D_{S,E})$*
*$E_T=\tt{EColoring}(T,D_{T,E})$* *\label{LineSetup2}*
for *$p$* in *$\tt{ColoredIsoms}(E_S,E_T)$*: *\label{LineMainLoop1}*
	if *$p(D_{S,E}) == D_{T,E}$*: 
		ToCp = *$[\textup{ }]$* *\label{LineExtension1}*
		for *$X$* in *$D_{S,E}$*:
			*$G = \texttt{MaxSubgp}(X)$*
			for *$j$* in *$X$*:
				for *$k$* in *$X$*:
					if *$j$* == *$k$*:
						ToCp.append(*$G$*.automorphisms*$()$*)
					else:
						ToCp.append(*$\texttt{Bijections}(G,G)$*)
		for T in *$\texttt{CartesianProduct}($*ToCp*$)$*: *\label{InnerLoop}*
			d = *$\tt{dict}([\up{ }])$*
			i = 0
			for *$X$* in *$D_{S,E}$*:
				for *$j$* in *$X$*:
					for *$k$* in *$X$*:
						for *$g$* in T[i]:
							d[*$\overline{g_{j,k}}$*] = *$\overline{(T[i](g))_{p(j),p(k)}}$*
						i = i+1
			if *\texttt{IsISGHomomorphism}*(d):
				return True *\label{LineExtension2}*
return False						
\end{lstlisting}
\end{alg}

\begin{prop}[Correctness of Algorithm \ref{AlgIsIsoc}] For (finite) inverse semigroups $S$ and $T$ generated by the $\texttt{ESN}$ function in Algorithm \ref{AlgMain} or \ref{AlgBasic}, if $E=E(S)=E(T$) and $I(S)=\texttt{Invariants}(S) = \texttt{Invariants}(T)=I(T)$, then $\texttt{IsIsoc}(S,T)$ returns True if $S\cong T$ and returns False otherwise.
\end{prop}

\begin{proof}Since $I(S)=I(T)$, $S$ and $T$ have the same number of lonely idempotents. Suppose the lonely idempotents of $S$ are $Sl_1 ,\ldots, Sl_j$ and the lonely idempotents of $T$ are $Tl_1 ,\ldots, Tl_j$. Lines \ref{LineSetup1}--\ref{LineSetup2} set up the loop beginning at line \ref{LineMainLoop1} to iterate over the automorphisms $p:E\rightarrow E$ such that 
\begin{itemize}
	\item for every $e\in E$, $G_e \cong G_{p(e)}$,
	\item for every $e\in E$, $|D_{S,E}(e)| = |D_{T,E}(p(e))|$, and
	\item $p(Sl_i) = Tl_i$ for all $i\in \{1 ,\ldots, j\}$.
\end{itemize}

By Theorem \ref{ThmLonely} and parts \eqref{Isom0}, \eqref{Isom1}, \eqref{Isom2} and \eqref{Isom5} of Theorem \ref{ThmIsomTest}, $S \cong T$ if and only if there exists an isomorphism $d:S\rightarrow T$ extending some such $p$. Furthermore, by parts \eqref{Isom0} and \eqref{Isom3} of Theorem \ref{ThmIsomTest}, we only need to consider extensions of $p$ for which we have, as sets, $p(D_{S,E}) = D_{T,E}$. The loop on lines \ref{InnerLoop}--\ref{LineExtension2} searches for an isomorphism $d$ extending such a $p$. In particular, this loop checks every extension of such a $p$ to a bijection $d:S\rightarrow T$ such that 
\begin{itemize}
	\item for each $e\in E$, $d|_{G_e\subseteq S}:G_e \rightarrow G_{d(e)}\subseteq T$ is an isomorphism of groups,
	\item for each $s\in S$, $d(\ran(s)) = p(\ran(s)) = \ran(d(s))$, and 
	\item for each $s\in S$, $d(\dom(s)) = p(\dom(s)) = \dom(d(s))$.
\end{itemize}
By parts \eqref{Isom1} and \eqref{Isom4} of Theorem \ref{ThmIsomTest}, $S\cong T$ if and only if for some $p$, some such $d$ is a homomorphism.
\end{proof}

\section{The inverse semigroups of order $\leq 15$}

\label{SecResults}


If $E$ is a meet-semilattice of order $n$, then up to isomorphism the only inverse semigroup $S$ of order $n$ such that $E(S)=E$ is $S=E$ itself. Therefore to count the number of inverse semigroups of order $n$ it suffices to iterate over $m\in \{1 ,\ldots, n-1\}$ on line \ref{OneThruNStep} of Algorithm \ref{AlgMain}, count the number of inverse semigroups output by the algorithm, and add the result to the number of meet-semilattices of order $n$.

We have implemented Algorithm \ref{AlgMain} and have used our implementation to count the inverse semigroups of order $1$ through $15$. Along the way we also counted the number of commutative inverse semigroups, inverse monoids, and commutative inverse monoids. These counts were given in Table \ref{TableISGs} in Section \ref{SecIntro}. 

Tables \ref{ISGOrder2}--\ref{ISGOrder15} contain more detailed information. In these tables ``ISGs" stands for inverse semigroups, ``IMs" stands for inverse monoids, and ``Comm." stands for commutative. We report in these tables the number of inverse semigroups, commutative inverse semigroups, inverse monoids, and commutative inverse monoids $S$ of order $2$ through $15$, broken down by number $|E(S)|$ of idempotents and the shape of the set partition $D_{S,E}$ of $E(S)$ given by restricting Green's $\D$-relation on $S$ to $E(S)$. In particular, given $n$, a number $m$ of idempotents, and a partition $\lambda$ of $m$, an entry of the form $X//Y$ in the ISGs//Semilattices column of Table $n$ indicates that there are $Y$ meet-semilattices $E$ of order $m$ for which there exists an inverse semigroup $S$ of order $n$ such that $E(S)=E$ and the restriction of $\D$ on $S$ to $E$ has shape $\lambda$, and that there are $X$ such inverse semigroups. The pairs of numbers in the other columns have analogous interpretations. Cells are left blank if their entries are $0//0$.

If $E$ is a meet-semilattice of order $m\leq n$, then there is always at least one inverse semigroup $S$ of order $n$ such that $E(S)=E$ and such that the shape of the restriction of $\D$ on $S$ to $E$ is the all-ones partition, so the value $Y$ in the $X//Y$ pair in the all-ones partition portion of the $m$th row and the ISGs//semilattices (resp.\ IMs//lattices) column is just the number of meet-semilattices (resp.\ lattices) of order $m$. 

To avoid writing repeated ones in our partitions, we use the notation $1_j$ to indicate $j$ repeated ones. So, for instance, we write $(3,3,2,1,1,1,1) = (3,3,2,1_4)$ and $(1,1,1,1,1)=(1_5)$.



\begin{table}[ht]
\btablesize
\caption{The inverse semigroups of order 2}
\begin{center}
\begin{tabular}{|c|c|c|c|c|c|}
\hline
\multirow{2}{*}{Idempotents} & Shape & {ISGs //} & {Comm. ISGs // } & {IMs // } & {Comm. IMs //}  \\
 & of $D_{S,E}$ &semilattices&semilattices&lattices&lattices\\
\hline
\multirow{1}{*}{1} & $(1)$ & 1//1 & 1//1 & 1//1 & 1//1 \\
\hline
\multirow{1}{*}{2} & $(1_{2})$ & 1//1 & 1//1 & 1//1 & 1//1\\
\hline
\multicolumn{2}{|r|}{Semigroup totals} & 2 & 2 & 2 & 2 \\
\cline{1-6}
\end{tabular}
\label{ISGOrder2}
\end{center}
\etablesize
\end{table}

\begin{table}[ht]
\btablesize
\caption{The inverse semigroups of order 3}
\begin{center}
\begin{tabular}{|c|c|c|c|c|c|}
\hline
\multirow{2}{*}{Idempotents} & Shape & {ISGs //} & {Comm. ISGs // } & {IMs // } & {Comm. IMs //}  \\
 & of $D_{S,E}$&semilattices&semilattices&lattices&lattices\\
\hline
\multirow{1}{*}{1} & $(1)$ & 1//1 & 1//1 & 1//1 & 1//1 \\
\hline
\multirow{1}{*}{2} & $(1_{2})$ & 2//1 & 2//1 & 2//1 & 2//1 \\
\hline
\multirow{1}{*}{3} & $(1_{3})$ & 2//2 & 2//2 & 1//1 & 1//1\\
\hline
\multicolumn{2}{|r|}{Semigroup totals} & 5 & 5 & 4 & 4 \\
\cline{1-6}
\end{tabular}
\label{ISGOrder3}
\end{center}
\etablesize
\end{table}

\begin{table}[ht]
\btablesize
\caption{The inverse semigroups of order 4}
\begin{center}
\begin{tabular}{|c|c|c|c|c|c|}
\hline
\multirow{2}{*}{Idempotents} & Shape & {ISGs //} & {Comm. ISGs // } & {IMs // } & {Comm. IMs //}  \\
 & of $D_{S,E}$&semilattices&semilattices&lattices&lattices\\
\hline
\multirow{1}{*}{1} & $(1)$ & 2//1 & 2//1 & 2//1 & 2//1 \\
\hline
\multirow{1}{*}{2} & $(1_{2})$ & 4//1 & 4//1 & 4//1 & 4//1 \\
\hline
\multirow{1}{*}{3} & $(1_{3})$ & 5//2 & 5//2 & 3//1 & 3//1 \\
\hline
\multirow{1}{*}{4} & $(1_{4})$ & 5//5 & 5//5 & 2//2 & 2//2\\
\hline
\multicolumn{2}{|r|}{Semigroup totals} & 16 & 16 & 11 & 11 \\
\cline{1-6}
\end{tabular}
\label{ISGOrder4}
\end{center}
\etablesize
\end{table}

\begin{table}[ht]
\btablesize
\caption{The inverse semigroups of order 5}
\begin{center}
\begin{tabular}{|c|c|c|c|c|c|}
\hline
\multirow{2}{*}{Idempotents} & Shape & {ISGs //} & {Comm. ISGs // } & {IMs // } & {Comm. IMs //}  \\
 & of $D_{S,E}$&semilattices&semilattices&lattices&lattices\\
\hline
\multirow{1}{*}{1} & $(1)$ & 1//1 & 1//1 & 1//1 & 1//1 \\
\hline
\multirow{1}{*}{2} & $(1_{2})$ & 6//1 & 6//1 & 6//1 & 6//1 \\
\hline
\multirow{2}{*}{3} & $(2,1)$ & 1//1 &  &  & \\
 & $(1_{3})$ & 13//2 & 13//2 & 8//1 & 8//1 \\
\hline
\multirow{1}{*}{4} & $(1_{4})$ & 16//5 & 16//5 & 7//2 & 7//2 \\
\hline
\multirow{1}{*}{5} & $(1_{5})$ & 15//15 & 15//15 & 5//5 & 5//5\\
\hline
\multicolumn{2}{|r|}{Semigroup totals} & 52 & 51 & 27 & 27 \\
\cline{1-6}
\end{tabular}
\label{ISGOrder5}
\end{center}
\etablesize
\end{table}

\begin{table}[ht]
\btablesize
\caption{The inverse semigroups of order 6}
\begin{center}
\begin{tabular}{|c|c|c|c|c|c|}
\hline
\multirow{2}{*}{Idempotents} & Shape & {ISGs //} & {Comm. ISGs // } & {IMs // } & {Comm. IMs //}  \\
 & of $D_{S,E}$&semilattices&semilattices&lattices&lattices\\
\hline
\multirow{1}{*}{1} & $(1)$ & 2//1 & 1//1 & 2//1 & 1//1 \\
\hline
\multirow{1}{*}{2} & $(1_{2})$ & 12//1 & 12//1 & 12//1 & 12//1 \\
\hline
\multirow{2}{*}{3} & $(2,1)$ & 2//1 &  &  & \\
 & $(1_{3})$ & 26//2 & 26//2 & 16//1 & 16//1 \\
\hline
\multirow{2}{*}{4} & $(2,1_{2})$ & 4//4 &  & 1//1 &  \\
 & $(1_{4})$ & 49//5 & 49//5 & 22//2 & 22//2 \\
\hline
\multirow{1}{*}{5} & $(1_{5})$ & 60//15 & 60//15 & 21//5 & 21//5 \\
\hline
\multirow{1}{*}{6} & $(1_{6})$ & 53//53 & 53//53 & 15//15 & 15//15\\
\hline
\multicolumn{2}{|r|}{Semigroup totals} & 208 & 201 & 89 & 87 \\
\cline{1-6}
\end{tabular}
\label{ISGOrder6}
\end{center}
\etablesize
\end{table}

\begin{table}[ht]
\btablesize
\caption{The inverse semigroups of order 7}
\begin{center}
\begin{tabular}{|c|c|c|c|c|c|}
\hline
\multirow{2}{*}{Idempotents} & Shape & {ISGs //} & {Comm. ISGs // } & {IMs // } & {Comm. IMs //}  \\
 & of $D_{S,E}$&semilattices&semilattices&lattices&lattices\\
\hline
\multirow{1}{*}{1} & $(1)$ & 1//1 & 1//1 & 1//1 & 1//1 \\
\hline
\multirow{1}{*}{2} & $(1_{2})$ & 10//1 & 8//1 & 10//1 & 8//1 \\
\hline
\multirow{2}{*}{3} & $(2,1)$ & 2//1 &  &  & \\
 & $(1_{3})$ & 51//2 & 51//2 & 33//1 & 33//1 \\
\hline
\multirow{2}{*}{4} & $(2,1_{2})$ & 13//4 &  & 4//1 &  \\
 & $(1_{4})$ & 118//5 & 118//5 & 54//2 & 54//2 \\
\hline
\multirow{2}{*}{5} & $(2,1_{3})$ & 17//14 &  & 4//4 &  \\
 & $(1_{5})$ & 215//15 & 215//15 & 76//5 & 76//5 \\
\hline
\multirow{1}{*}{6} & $(1_{6})$ & 262//53 & 262//53 & 75//15 & 75//15 \\
\hline
\multirow{1}{*}{7} & $(1_{7})$ & 222//222 & 222//222 & 53//53 & 53//53\\
\hline
\multicolumn{2}{|r|}{Semigroup totals} & 911 & 877 & 310 & 300 \\
\cline{1-6}
\end{tabular}
\label{ISGOrder7}
\end{center}
\etablesize
\end{table}

\begin{table}[ht]
\btablesize
\caption{The inverse semigroups of order 8}
\begin{center}
\begin{tabular}{|c|c|c|c|c|c|}
\hline
\multirow{2}{*}{Idempotents} & Shape & {ISGs //} & {Comm. ISGs // } & {IMs // } & {Comm. IMs //}  \\
 & of $D_{S,E}$&semilattices&semilattices&lattices&lattices\\
\hline
\multirow{1}{*}{1} & $(1)$ & 5//1 & 3//1 & 5//1 & 3//1 \\
\hline
\multirow{1}{*}{2} & $(1_{2})$ & 22//1 & 18//1 & 22//1 & 18//1 \\
\hline
\multirow{2}{*}{3} & $(2,1)$ & 5//1 &  &  & \\
 & $(1_{3})$ & 85//2 & 80//2 & 54//1 & 51//1 \\
\hline
\multirow{2}{*}{4} & $(2,1_{2})$ & 26//4 &  & 7//1 &  \\
 & $(1_{4})$ & 269//5 & 269//5 & 124//2 & 124//2 \\
\hline
\multirow{2}{*}{5} & $(2,1_{3})$ & 70//14 &  & 19//4 &  \\
 & $(1_{5})$ & 601//15 & 601//15 & 215//5 & 215//5 \\
\hline
\multirow{2}{*}{6} & $(2,1_{4})$ & 82//52 &  & 17//14 &  \\
 & $(1_{6})$ & 1079//53 & 1079//53 & 311//15 & 311//15 \\
\hline
\multirow{1}{*}{7} & $(1_{7})$ & 1315//222 & 1315//222 & 315//53 & 315//53 \\
\hline
\multirow{1}{*}{8} & $(1_{8})$ & 1078//1078 & 1078//1078 & 222//222 & 222//222\\
\hline
\multicolumn{2}{|r|}{Semigroup totals} & 4637 & 4443 & 1311 & 1259 \\
\cline{1-6}
\end{tabular}
\label{ISGOrder8}
\end{center}
\etablesize
\end{table}

\begin{table}[ht]
\btablesize
\caption{The inverse semigroups of order 9}
\begin{center}
\begin{tabular}{|c|c|c|c|c|c|}
\hline
\multirow{2}{*}{Idempotents} & Shape & {ISGs //} & {Comm. ISGs // } & {IMs // } & {Comm. IMs //}  \\
 & of $D_{S,E}$&semilattices&semilattices&lattices&lattices\\
\hline
\multirow{1}{*}{1} & $(1)$ & 2//1 & 2//1 & 2//1 & 2//1 \\
\hline
\multirow{1}{*}{2} & $(1_{2})$ & 23//1 & 16//1 & 23//1 & 16//1 \\
\hline
\multirow{2}{*}{3} & $(2,1)$ & 3//1 &  &  & \\
 & $(1_{3})$ & 126//2 & 111//2 & 82//1 & 72//1 \\
\hline
\multirow{2}{*}{4} & $(2,1_{2})$ & 47//4 &  & 14//1 &  \\
 & $(1_{4})$ & 520//5 & 504//5 & 245//2 & 238//2 \\
\hline
\multirow{3}{*}{5} & $(2,2,1)$ & 3//3 &  &  & \\
 & $(2,1_{3})$ & 192//14 &  & 53//4 &  \\
 & $(1_{5})$ & 1555//15 & 1555//15 & 562//5 & 562//5 \\
\hline
\multirow{2}{*}{6} & $(2,1_{4})$ & 410//52 &  & 92//14 &  \\
 & $(1_{6})$ & 3460//53 & 3460//53 & 1003//15 & 1003//15 \\
\hline
\multirow{2}{*}{7} & $(2,1_{5})$ & 445//221 &  & 82//52 &  \\
 & $(1_{7})$ & 6137//222 & 6137//222 & 1480//53 & 1480//53 \\
\hline
\multirow{1}{*}{8} & $(1_{8})$ & 7505//1078 & 7505//1078 & 1537//222 & 1537//222 \\
\hline
\multirow{1}{*}{9} & $(1_{9})$ & 5994//5994 & 5994//5994 & 1078//1078 & 1078//1078\\
\hline
\multicolumn{2}{|r|}{Semigroup totals} & 26422 & 25284 & 6253 & 5988 \\
\cline{1-6}
\end{tabular}
\label{ISGOrder9}
\end{center}
\etablesize
\end{table}

\begin{table}[ht]
\btablesize
\caption{The inverse semigroups of order 10}
\begin{center}
\begin{tabular}{|c|c|c|c|c|c|}
\hline
\multirow{2}{*}{Idempotents} & Shape & {ISGs //} & {Comm. ISGs // } & {IMs // } & {Comm. IMs //}  \\
 & of $D_{S,E}$&semilattices&semilattices&lattices&lattices\\
\hline
\multirow{1}{*}{1} & $(1)$ & 2//1 & 1//1 & 2//1 & 1//1 \\
\hline
\multirow{1}{*}{2} & $(1_{2})$ & 48//1 & 30//1 & 48//1 & 30//1 \\
\hline
\multirow{2}{*}{3} & $(2,1)$ & 10//1 &  &  & \\
 & $(1_{3})$ & 235//2 & 193//2 & 151//1 & 125//1 \\
\hline
\multirow{3}{*}{4} & $(3,1)$ & 1//1 &  &  & \\
 & $(2,1_{2})$ & 92//4 &  & 23//1 &  \\
 & $(1_{4})$ & 981//5 & 918//5 & 462//2 & 433//2 \\
\hline
\multirow{3}{*}{5} & $(2,2,1)$ & 7//3 &  &  & \\
 & $(2,1_{3})$ & 424//14 &  & 118//4 &  \\
 & $(1_{5})$ & 3499//15 & 3439//15 & 1273//5 & 1252//5 \\
\hline
\multirow{3}{*}{6} & $(2,2,1_{2})$ & 27//24 &  & 3//3 &  \\
 & $(2,1_{4})$ & 1387//52 &  & 321//14 &  \\
 & $(1_{6})$ & 10016//53 & 10016//53 & 2928//15 & 2928//15 \\
\hline
\multirow{2}{*}{7} & $(2,1_{5})$ & 2629//221 &  & 508//52 &  \\
 & $(1_{7})$ & 22254//222 & 22254//222 & 5389//53 & 5389//53 \\
\hline
\multirow{2}{*}{8} & $(2,1_{6})$ & 2704//1077 &  & 445//221 &  \\
 & $(1_{8})$ & 39164//1078 & 39164//1078 & 8077//222 & 8077//222 \\
\hline
\multirow{1}{*}{9} & $(1_{9})$ & 48061//5994 & 48061//5994 & 8583//1078 & 8583//1078 \\
\hline
\multirow{1}{*}{10} & $(1_{10})$ & 37622//37622 & 37622//37622 & 5994//5994 & 5994//5994\\
\hline
\multicolumn{2}{|r|}{Semigroup totals} & 169163 & 161698 & 34325 & 32812 \\
\cline{1-6}
\end{tabular}
\label{ISGOrder10}
\end{center}
\etablesize
\end{table}

\begin{table}[ht]
\btablesize
\caption{The inverse semigroups of order 11}
\begin{center}
\begin{tabular}{|c|c|c|c|c|c|}
\hline
\multirow{2}{*}{Idempotents} & Shape & {ISGs //} & {Comm. ISGs // } & {IMs // } & {Comm. IMs //}  \\
 & of $D_{S,E}$&semilattices&semilattices&lattices&lattices\\
\hline
\multirow{1}{*}{1} & $(1)$ & 1//1 & 1//1 & 1//1 & 1//1 \\
\hline
\multirow{1}{*}{2} & $(1_{2})$ & 26//1 & 18//1 & 26//1 & 18//1 \\
\hline
\multirow{2}{*}{3} & $(2,1)$ & 4//1 &  &  & \\
 & $(1_{3})$ & 301//2 & 215//2 & 198//1 & 141//1 \\
\hline
\multirow{3}{*}{4} & $(3,1)$ & 2//1 &  &  & \\
 & $(2,1_{2})$ & 113//4 &  & 34//1 &  \\
 & $(1_{4})$ & 1707//5 & 1495//5 & 808//2 & 710//2 \\
\hline
\multirow{4}{*}{5} & $(3,1_{2})$ & 5//5 &  & 1//1 &  \\
 & $(2,2,1)$ & 7//3 &  &  & \\
 & $(2,1_{3})$ & 904//14 &  & 253//4 &  \\
 & $(1_{5})$ & 7407//15 & 7108//15 & 2723//5 & 2616//5 \\
\hline
\multirow{3}{*}{6} & $(2,2,1_{2})$ & 105//24 &  & 14//3 &  \\
 & $(2,1_{4})$ & 3660//52 &  & 866//14 &  \\
 & $(1_{6})$ & 25503//53 & 25241//53 & 7507//15 & 7432//15 \\
\hline
\multirow{3}{*}{7} & $(2,2,1_{3})$ & 216//149 &  & 27//24 &  \\
 & $(2,1_{5})$ & 10518//221 &  & 2085//52 &  \\
 & $(1_{7})$ & 71439//222 & 71439//222 & 17439//53 & 17439//53 \\
\hline
\multirow{2}{*}{8} & $(2,1_{6})$ & 18510//1077 &  & 3134//221 &  \\
 & $(1_{8})$ & 158478//1078 & 158478//1078 & 32845//222 & 32845//222 \\
\hline
\multirow{2}{*}{9} & $(2,1_{7})$ & 18232//5993 &  & 2704//1077 &  \\
 & $(1_{9})$ & 277347//5994 & 277347//5994 & 49905//1078 & 49905//1078 \\
\hline
\multirow{1}{*}{10} & $(1_{10})$ & 341390//37622 & 341390//37622 & 54055//5994 & 54055//5994 \\
\hline
\multirow{1}{*}{11} & $(1_{11})$ & 262776//262776 & 262776//262776 & 37622//37622 & 37622//37622\\
\hline
\multicolumn{2}{|r|}{Semigroup totals} & 1198651 & 1145508 & 212247 & 202784 \\
\cline{1-6}
\end{tabular}
\label{ISGOrder11}
\end{center}
\etablesize
\end{table}

\begin{table}[ht]
\btablesize
\caption{The inverse semigroups of order 12}
\begin{center}
\begin{tabular}{|c|c|c|c|c|c|}
\hline
\multirow{2}{*}{Idempotents} & Shape & {ISGs //} & {Comm. ISGs // } & {IMs // } & {Comm. IMs //}  \\
 & of $D_{S,E}$&semilattices&semilattices&lattices&lattices\\
\hline
\multirow{1}{*}{1} & $(1)$ & 5//1 & 2//1 & 5//1 & 2//1 \\
\hline
\multirow{1}{*}{2} & $(1_{2})$ & 93//1 & 56//1 & 93//1 & 56//1 \\
\hline
\multirow{2}{*}{3} & $(2,1)$ & 26//1 &  &  & \\
 & $(1_{3})$ & 544//2 & 367//2 & 349//1 & 236//1 \\
\hline
\multirow{3}{*}{4} & $(3,1)$ & 3//1 &  &  & \\
 & $(2,1_{2})$ & 227//4 &  & 59//1 &  \\
 & $(1_{4})$ & 3081//5 & 2535//5 & 1473//2 & 1220//2 \\
\hline
\multirow{4}{*}{5} & $(3,1_{2})$ & 19//5 &  & 4//1 &  \\
 & $(2,2,1)$ & 20//3 &  &  & \\
 & $(2,1_{3})$ & 1650//14 &  & 466//4 &  \\
 & $(1_{5})$ & 14725//15 & 13552//15 & 5430//5 & 5010//5 \\
\hline
\multirow{4}{*}{6} & $(3,1_{3})$ & 26//23 &  & 5//5 &  \\
 & $(2,2,1_{2})$ & 209//24 &  & 25//3 &  \\
 & $(2,1_{4})$ & 8865//52 &  & 2118//14 &  \\
 & $(1_{6})$ & 60352//53 & 58761//53 & 17905//15 & 17444//15 \\
\hline
\multirow{3}{*}{7} & $(2,2,1_{3})$ & 1078//149 &  & 151//24 &  \\
 & $(2,1_{5})$ & 32320//221 &  & 6546//52 &  \\
 & $(1_{7})$ & 202397//222 & 201082//222 & 49742//53 & 49427//53 \\
\hline
\multirow{3}{*}{8} & $(2,2,1_{4})$ & 1780//883 &  & 216//149 &  \\
 & $(2,1_{6})$ & 85146//1077 &  & 14739//221 &  \\
 & $(1_{8})$ & 559264//1078 & 559264//1078 & 116857//222 & 116857//222 \\
\hline
\multirow{2}{*}{9} & $(2,1_{7})$ & 142296//5993 &  & 21476//1077 &  \\
 & $(1_{9})$ & 1237965//5994 & 1237965//5994 & 224095//1078 & 224095//1078 \\
\hline
\multirow{2}{*}{10} & $(2,1_{8})$ & 135249//37621 &  & 18232//5993 &  \\
 & $(1_{10})$ & 2157481//37622 & 2157481//37622 & 344406//5994 & 344406//5994 \\
\hline
\multirow{1}{*}{11} & $(1_{11})$ & 2660921//262776 & 2660921//262776 & 379012//37622 & 379012//37622 \\
\hline
\multirow{1}{*}{12} & $(1_{12})$ & 2018305//2018305 & 2018305//2018305 & 262776//262776 & 262776//262776\\
\hline
\multicolumn{2}{|r|}{Semigroup totals} & 9324047 & 8910291 & 1466180 & 1400541 \\
\cline{1-6}
\end{tabular}
\label{ISGOrder12}
\end{center}
\etablesize
\end{table}

\begin{table}[ht]
\btablesize
\caption{The inverse semigroups of order 13}
\begin{center}
\begin{tabular}{|c|c|c|c|c|c|}
\hline
\multirow{2}{*}{Idempotents} & Shape & {ISGs //} & {Comm. ISGs // } & {IMs // } & {Comm. IMs //}  \\
 & of $D_{S,E}$&semilattices&semilattices&lattices&lattices\\
\hline
\multirow{1}{*}{1} & $(1)$ & 1//1 & 1//1 & 1//1 & 1//1 \\
\hline
\multirow{1}{*}{2} & $(1_{2})$ & 38//1 & 24//1 & 38//1 & 24//1 \\
\hline
\multirow{2}{*}{3} & $(2,1)$ & 8//1 &  &  & \\
 & $(1_{3})$ & 634//2 & 412//2 & 419//1 & 272//1 \\
\hline
\multirow{3}{*}{4} & $(3,1)$ & 7//1 &  &  & \\
 & $(2,1_{2})$ & 295//4 &  & 88//1 &  \\
 & $(1_{4})$ & 4717//5 & 3479//5 & 2246//2 & 1660//2 \\
\hline
\multirow{4}{*}{5} & $(3,1_{2})$ & 44//5 &  & 9//1 &  \\
 & $(2,2,1)$ & 14//3 &  &  & \\
 & $(2,1_{3})$ & 2777//14 &  & 786//4 &  \\
 & $(1_{5})$ & 28025//15 & 24490//15 & 10385//5 & 9106//5 \\
\hline
\multirow{4}{*}{6} & $(3,1_{3})$ & 134//23 &  & 27//5 &  \\
 & $(2,2,1_{2})$ & 428//24 &  & 55//3 &  \\
 & $(2,1_{4})$ & 18873//52 &  & 4592//14 &  \\
 & $(1_{6})$ & 132846//53 & 125672//53 & 39675//15 & 37592//15 \\
\hline
\multirow{5}{*}{7} & $(3,1_{4})$ & 153//117 &  & 26//23 &  \\
 & $(2,2,2,1)$ & 13//12 &  &  & \\
 & $(2,2,1_{3})$ & 3063//149 &  & 441//24 &  \\
 & $(2,1_{5})$ & 88364//221 &  & 18171//52 &  \\
 & $(1_{7})$ & 528405//222 & 518948//222 & 130955//53 & 128668//53 \\
\hline
\multirow{3}{*}{8} & $(2,2,1_{4})$ & 10719//883 &  & 1385//149 &  \\
 & $(2,1_{6})$ & 298708//1077 &  & 52723//221 &  \\
 & $(1_{8})$ & 1741789//1078 & 1734284//1078 & 366740//222 & 365203//222 \\
\hline
\multirow{3}{*}{9} & $(2,2,1_{5})$ & 15456//5435 &  & 1780//883 &  \\
 & $(2,1_{7})$ & 737996//5993 &  & 113535//1077 &  \\
 & $(1_{9})$ & 4764281//5994 & 4764281//5994 & 869969//1078 & 869969//1078 \\
\hline
\multirow{2}{*}{10} & $(2,1_{8})$ & 1187056//37621 &  & 161843//5993 &  \\
 & $(1_{10})$ & 10518061//37622 & 10518061//37622 & 1691090//5994 & 1691090//5994 \\
\hline
\multirow{2}{*}{11} & $(2,1_{9})$ & 1093871//262775 &  & 135249//37621 &  \\
 & $(1_{11})$ & 18265468//262776 & 18265468//262776 & 2623757//37622 & 2623757//37622 \\
\hline
\multirow{1}{*}{12} & $(1_{12})$ & 22545079//2018305 & 22545079//2018305 & 2923697//262776 & 2923697//262776 \\
\hline
\multirow{1}{*}{13} & $(1_{13})$ & 16873364//16873364 & 16873364//16873364 & 2018305//2018305 & 2018305//2018305\\
\hline
\multicolumn{2}{|r|}{Semigroup totals} & 78860687 & 75373563 & 11167987 & 10669344 \\
\cline{1-6}
\end{tabular}
\label{ISGOrder13}
\end{center}
\etablesize
\end{table}

\begin{table}[ht]
\btablesize
\caption{The inverse semigroups of order 14}
\begin{center}
\begin{tabular}{|c|c|c|c|c|c|}
\hline
\multirow{2}{*}{Idempotents} & Shape & {ISGs //} & {Comm. ISGs // } & {IMs // } & {Comm. IMs //}  \\
 & of $D_{S,E}$&semilattices&semilattices&lattices&lattices\\
\hline
\multirow{1}{*}{1} & $(1)$ & 2//1 & 1//1 & 2//1 & 1//1 \\
\hline
\multirow{1}{*}{2} & $(1_{2})$ & 95//1 & 40//1 & 95//1 & 40//1 \\
\hline
\multirow{2}{*}{3} & $(2,1)$ & 20//1 &  &  & \\
 & $(1_{3})$ & 1225//2 & 706//2 & 801//1 & 465//1 \\
\hline
\multirow{3}{*}{4} & $(3,1)$ & 3//1 &  &  & \\
 & $(2,1_{2})$ & 629//4 &  & 175//1 &  \\
 & $(1_{4})$ & 8460//5 & 5977//5 & 4071//2 & 2899//2 \\
\hline
\multirow{4}{*}{5} & $(3,1_{2})$ & 91//5 &  & 19//1 &  \\
 & $(2,2,1)$ & 51//3 &  &  & \\
 & $(2,1_{3})$ & 5309//14 &  & 1510//4 &  \\
 & $(1_{5})$ & 51551//15 & 42161//15 & 19261//5 & 15848//5 \\
\hline
\multirow{5}{*}{6} & $(3,2,1)$ & 7//6 &  &  & \\
 & $(3,1_{3})$ & 418//23 &  & 87//5 &  \\
 & $(2,2,1_{2})$ & 816//24 &  & 93//3 &  \\
 & $(2,1_{4})$ & 37344//52 &  & 9176//14 &  \\
 & $(1_{6})$ & 278911//53 & 254127//53 & 83827//15 & 76586//15 \\
\hline
\multirow{5}{*}{7} & $(3,1_{4})$ & 976//117 &  & 170//23 &  \\
 & $(2,2,2,1)$ & 32//12 &  &  & \\
 & $(2,2,1_{3})$ & 7216//149 &  & 1045//24 &  \\
 & $(2,1_{5})$ & 213876//221 &  & 44648//52 &  \\
 & $(1_{7})$ & 1279242//222 & 1230949//222 & 319361//53 & 307637//53 \\
\hline
\multirow{5}{*}{8} & $(3,1_{5})$ & 999//653 &  & 153//117 &  \\
 & $(2,2,2,1_{2})$ & 240//191 &  & 13//12 &  \\
 & $(2,2,1_{4})$ & 38341//883 &  & 5144//149 &  \\
 & $(2,1_{6})$ & 912857//1077 &  & 163837//221 &  \\
 & $(1_{8})$ & 4967113//1078 & 4904704//1078 & 1055099//222 & 1042206//222 \\
\hline
\multirow{3}{*}{9} & $(2,2,1_{5})$ & 108619//5435 &  & 12998//883 &  \\
 & $(2,1_{7})$ & 2908054//5993 &  & 455377//1077 &  \\
 & $(1_{9})$ & 16156724//5994 & 16108663//5994 & 2975421//1078 & 2966838//1078 \\
\hline
\multirow{3}{*}{10} & $(2,2,1_{6})$ & 142385//35893 &  & 15456//5435 &  \\
 & $(2,1_{8})$ & 6838144//37621 &  & 948366//5993 &  \\
 & $(1_{10})$ & 43822653//37622 & 43822653//37622 & 7111831//5994 & 7111831//5994 \\
\hline
\multirow{2}{*}{11} & $(2,1_{9})$ & 10673677//262775 &  & 1329810//37621 &  \\
 & $(1_{11})$ & 96447794//262776 & 96447794//262776 & 13966574//37622 & 13966574//37622 \\
\hline
\multirow{2}{*}{12} & $(2,1_{10})$ & 9569171//2018304 &  & 1093871//262775 &  \\
 & $(1_{12})$ & 166932647//2018305 & 166932647//2018305 & 21834255//262776 & 21834255//262776 \\
\hline
\multirow{1}{*}{13} & $(1_{13})$ & 205966795//16873364 & 205966795//16873364 & 24563384//2018305 & 24563384//2018305 \\
\hline
\multirow{1}{*}{14} & $(1_{14})$ & 152233518//152233518 & 152233518//152233518 & 16873364//16873364 & 16873364//16873364\\
\hline
\multicolumn{2}{|r|}{Semigroup totals} & 719606005 & 687950735 & 92889294 & 88761928 \\
\cline{1-6}
\end{tabular}
\label{ISGOrder14}
\end{center}
\etablesize
\end{table}

\begin{table}[ht]
\btablesize
\caption{The inverse semigroups of order 15}
\begin{center}
\begin{tabular}{|c|c|c|c|c|c|}
\hline
\multirow{2}{*}{Idempotents}& Shape & {ISGs //} & {Comm. ISGs // } & {IMs // } & {Comm. IMs //}  \\
 & of $D_{S,E}$&semilattices&semilattices&lattices&lattices\\
\hline
\multirow{1}{*}{1} & $(1)$ & 1//1 & 1//1 & 1//1 & 1//1 \\
\hline
\multirow{1}{*}{2} & $(1_{2})$ & 59//1 & 34//1 & 59//1 & 34//1 \\
\hline
\multirow{2}{*}{3} & $(2,1)$ & 7//1 &  &  & \\
 & $(1_{3})$ & 1017//2 & 580//2 & 672//1 & 382//1 \\
\hline
\multirow{3}{*}{4} & $(3,1)$ & 12//1 &  &  & \\
 & $(2,1_{2})$ & 445//4 &  & 140//1 &  \\
 & $(1_{4})$ & 11963//5 & 7588//5 & 5734//2 & 3650//2 \\
\hline
\multirow{4}{*}{5} & $(3,1_{2})$ & 164//5 &  & 29//1 &  \\
 & $(2,2,1)$ & 22//3 &  &  & \\
 & $(2,1_{3})$ & 8202//14 &  & 2378//4 &  \\
 & $(1_{5})$ & 89791//15 & 68092//15 & 33434//5 & 25468//5 \\
\hline
\multirow{5}{*}{6} & $(3,2,1)$ & 18//6 &  &  & \\
 & $(3,1_{3})$ & 1020//23 &  & 211//5 &  \\
 & $(2,2,1_{2})$ & 1266//24 &  & 159//3 &  \\
 & $(2,1_{4})$ & 71562//52 &  & 17608//14 &  \\
 & $(1_{6})$ & 557310//53 & 482754//53 & 168321//15 & 146380//15 \\
\hline
\multirow{6}{*}{7} & $(3,2,1_{2})$ & 86//66 &  & 7//6 &  \\
 & $(3,1_{4})$ & 3741//117 &  & 671//23 &  \\
 & $(2,2,2,1)$ & 32//12 &  &  & \\
 & $(2,2,1_{3})$ & 15957//149 &  & 2316//24 &  \\
 & $(2,1_{5})$ & 476897//221 &  & 100989//52 &  \\
 & $(1_{7})$ & 2928371//222 & 2739985//222 & 736701//53 & 690700//53 \\
\hline
\multirow{5}{*}{8} & $(3,1_{5})$ & 7561//653 &  & 1168//117 &  \\
 & $(2,2,2,1_{2})$ & 1020//192 &  & 66//12 &  \\
 & $(2,2,1_{4})$ & 107023//883 &  & 14722//149 &  \\
 & $(2,1_{6})$ & 2467528//1077 &  & 449724//221 &  \\
 & $(1_{8})$ & 13101797//1078 & 12745673//1078 & 2806507//222 & 2732450//222 \\
\hline
\multirow{5}{*}{9} & $(3,1_{6})$ & 7225//4049 &  & 999//653 &  \\
 & $(2,2,2,1_{3})$ & 3381//2062 &  & 240//191 &  \\
 & $(2,2,1_{5})$ & 462898//5435 &  & 57163//883 &  \\
 & $(2,1_{7})$ & 9837747//5993 &  & 1566722//1077 &  \\
 & $(1_{9})$ & 49918237//5994 & 49464586//5994 & 9280078//1078 & 9198343//1078 \\
\hline
\multirow{3}{*}{10} & $(2,2,1_{6})$ & 1142433//35893 &  & 127151//5435 &  \\
 & $(2,1_{8})$ & 29864170//37621 &  & 4210461//5993 &  \\
 & $(1_{10})$ & 160561088//37622 & 160219698//37622 & 26299900//5994 & 26245845//5994 \\
\hline
\multirow{3}{*}{11} & $(2,2,1_{7})$ & 1390467//257001 &  & 142385//35893 &  \\
 & $(2,1_{9})$ & 67509604//262775 &  & 8541243//37621 &  \\
 & $(1_{11})$ & 432247509//262776 & 432247509//262776 & 63212608//37622 & 63212608//37622 \\
\hline
\multirow{2}{*}{12} & $(2,1_{10})$ & 102805707//2018304 &  & 11815609//262775 &  \\
 & $(1_{12})$ & 948037628//2018305 & 948037628//2018305 & 125084221//262776 & 125084221//262776 \\
\hline
\multirow{2}{*}{13} & $(2,1_{11})$ & 89902414//16873363 &  & 9569171//2018304 &  \\
 & $(1_{13})$ & 1635389858//16873364 & 1635389858//16873364 & 196698551//2018305 & 196698551//2018305 \\
\hline
\multirow{1}{*}{14} & $(1_{14})$ & 2014968017//152233518 & 2014968017//152233518 & 222840159//16873364 & 222840159//16873364 \\
\hline
\multirow{1}{*}{15} & $(1_{15})$ & 1471613387//1471613387 & 1471613387//1471613387 & 152233518//152233518 & 152233518//152233518\\
\hline
\multicolumn{2}{|r|}{Semigroup totals} & 7035514642 & 6727985390 & 836021796 & 799112310 \\
\cline{1-6}
\end{tabular}
\label{ISGOrder15}
\end{center}
\etablesize
\end{table}

We parallelized our implementation of Algorithm \ref{AlgMain} at line \ref{MainStepE} by spawning a new thread to carry out the computations for each meet-semilattice $E$. One benefit of our approach is that the parallelized threads do not need to communicate with one another, so the work required by the algorithm is easily distributable across several CPUs and/or computer systems if necessary. 

We ran our implementation on a computational server hosted at Sam Houston State University, which has four AMD Opteron\textsuperscript{TM} 6272 processors (a total of 64 cores), with each core running at 2.1 GHz, and 256 GB of RAM. 
The following running times are given in terms of the computational power of one core of our server. Including time spent testing for commutativity and counting monoids and inverse monoids along the way, our algorithm took a total of 11.1 CPU years to count the inverse semigroups of order $1 ,\ldots, 15$. Approximately 20\% of this time was spent on isomorphism tests. 92\% of this time was spent on $n=15$. 
We estimate that it would take approximately 100 CPU years for our implementation to count the inverse semigroups of order $16$. 

Thanks to our implementation of $\tt{Invariants}$ in Section \ref{SecInvariants}, of the 6201659106 inverse semigroups of order $1 ,\ldots, 15$ generated by our algorithm, 4317895179 of them (69.62\%) were accepted as new inverse semigroups immediately (with no isomorphism testing), and 2824933733 of those (65.42\%) were never involved in an isomorphism test. A total of 5491416345 isomorphism tests were performed, for an overall rate of 0.885 isomorphism tests per generated inverse semigroup. Statistics regarding the effectiveness of our implementation of $\tt{Invariants}$, broken down by $n$, are given in Table \ref{TableEfficacyOfInvariants}. These statistics show that while our implementation of $\tt{Invariants}$ becomes less effective as $n$ grows, it remains highly effective for all $n\leq 15$. 

\begin{table}%
\caption{Effectiveness of \tt{Invariants}}
\begin{tabular}{|c|c|c|c|}
\hline
\multirow{2}{*} n & {\% of generated ISGs} & \% of these never involved & \#isomorphism tests done / \\
 &  accepted immediately &  in isomorphism test & \#generated ISGs \\
 \hline
2 & 100\% & 100\% & 0.0 \\
3 & 100\% & 100\% & 0.0 \\
4 & 100\% & 90.9\% & 0.091 \\
5 & 100\% & 83.8\% & 0.189 \\
6 & 97.4\% & 75.5\% & 0.316 \\
7 & 94.3\% & 72.3\% & 0.409 \\
8 & 90.0\% & 69.9\% & 0.500 \\
9 & 85.9\% & 68.7\% & 0.559 \\
10 & 81.9\% & 67.8\% & 0.618 \\
11 & 78.5\% & 67.2\% & 0.668 \\
12 & 75.6\% & 66.7\% & 0.722 \\
13 & 73.2\% & 66.3\% & 0.776 \\
14 & 71.1\% & 65.8\% & 0.833 \\
15 & 69.4\% & 65.4\% & 0.892 \\
\hline
\end{tabular}
\label{TableEfficacyOfInvariants}
\end{table}



Although it is impossible to certify that our implementation of our algorithm (which consists of thousands of lines of \texttt{Sage} code) is bug-free and ran without error, all of the evidence we have points in this direction. First, our implementation correctly computed the number of inverse semigroups, commutative inverse semigroups, inverse monoids, and commutative inverse monoids of order $n$ for all previously-known values of $n$ ($n=1 ,\ldots, 9$), and agrees with the output of Distler's code \cite{DistlerThesis} for $n=10$. Second, to guard against system errors unrelated to our implementation that could nevertheless affect its output, we ran our program multiple times on multiple systems, including our server, for $n\leq 13$. We ran our program on our server multiple times for $n=14$ and twice for $n=15$, and we obtained the same results every time. Finally, in search of greater speed and memory efficiency, over the course of this project we recoded, entirely from scratch, several key steps of our algorithm in a number of different ways, and we obtained the same results regardless of which of our implementations of these key steps we used. 

\clearpage

\section{Additional proofs}

In this section we prove Theorem \ref{ThmESN2} and we prove the correctness of our implementation of $\texttt{GPosets}$ in Section \ref{SecPartialOrders}.

\subsection{Proof of Theorem \ref{ThmESN2}} 
\label{SecESN2Pf}
In this section we prove Theorem \ref{ThmESN2}. Our proof is essentially an elaboration of the main idea in Section 4 of \cite{Steinberg2}.

\begin{proof}[Proof of Theorem \ref{ThmESN2}] Let $S$ be a finite inverse semigroup. Recall that the semigroup algebra $\C S$ is a $\C$-vector space with basis $\{s\}_{s\in S}$, where multiplication is given by the extension of the multiplication in $S$ via the distributive law. Steinberg defines another basis (called the {\em groupoid basis}) $\{\ld s \rd\}_{s\in S}$ of $\C S$ as follows \cite{Steinberg2}. For $s\in S$, let
\[
\ld s \rd = \sum_{t\in S:t\leq s} \mu(t,s)t,
\]
where $\mu$ is the M\"obius function of the natural partial order $\leq$ on $S$. 


The groupoid basis multiplies in the following manner. For $s,t\in S$,
\begin{equation}\label{eqGroupoidMult}
\ld s \rd \ld t \rd = 
\begin{cases}
\ld st \rd &\up{if }\dom(s)=\ran(t);\\
0 &\up{otherwise}.
\end{cases}
\end{equation}
The groupoid basis is thus a basis of $\C S$, whose elements multiply as in a groupoid (where we interpret 0 as undefined). Of course, the natural basis of $\C S$ can be recovered by M\"obius inversion. Specifically, for $s\in S$, in $\C S$ we have
\[
s=\sum_{t\in S:t\leq s}\ld t \rd.
\]
The natural partial order of $S$ gives rise to a partial order on the groupoid basis: For  $s,t\in S$, let
\begin{equation*}
 \ld s \rd \leq \ld t \rd \iff s\leq t.
\end{equation*}
With this notation we can write $s$ in terms of the groupoid basis and its partial order:
\begin{equation}
\label{eqZetaTransform}
s=\sum_{t\in S:\ld t\rd \leq \ld s\rd}\ld t \rd.
\end{equation}
Note that the semilattice $E(S)$ is isomorphic to the semilattice $(\{\ld e \rd : e\in E(S)\},\leq)$ by $e\mapsto \ld e \rd$.

Now suppose the partition of $E(S)$ obtained by restricting Green's $\D$-relation on $S$ to $E(S)$ is $\{X_1 ,\ldots, X_k\}$. Suppose $|X_i|=r_i$ and $G_{X_i}\cong G_i$ for all $i\in \{1 ,\ldots, k\}$. Denote the $\D$-class of $S$ containing $X_i$ by $D_i$, and denote the $\C$-span of $H_i=\{\ld s\rd:s\in D_i\}$ by $\C H_i$. It is clear from (\ref{eqGroupoidMult}) that, as an algebra, $\C S=\bigoplus_{i=1}^k \C H_i$. For each $\D$-class $D_i$, fix an idempotent $e_i$, so $G_{e_i}\cong G_i$ for all $i$. 

Steinberg gives the following explicit algebra isomorphism from $\C H_i$ to $M_{r_i}(\C G_{e_i})$. For each $e\in X_i$, fix an element $p_e\in S$ such that $\dom(p_e)=e_i$ and $\ran(p_e)=e$, taking $p_{e_i}=e_i$. Note that $p_e\in D_i$, so $p_e^{-1}\in D_i$ as well. Viewing $r_i\times r_i$ matrices as being indexed by pairs of elements of $X_i$, define a map $\Phi_i:H_i\rightarrow M_{r_i}(\C G_{e_i})$ by
\[
\Phi_i(\ld s\rd)={p_{\ran(s)}}^{-1} s p_{\dom(s)} E_{\ran(s),\dom(s)},
\]
where $E_{\ran(s),\dom(s)}$ is the standard $r_i\times r_i$ matrix with a 1 in the $\ran(s),\dom(s)$ position and $0$ elsewhere. The linear extension of $\Phi_i$ to $\C H_i$ is Steinberg's isomorphism, with inverse induced by, for $s\in G_{e_i}$,
\[
sE_{e,f} \mapsto \ld p_e s p_f^{-1}\rd.
\]
Note that ${p_{\ran(s)}}^{-1} s p_{\dom(s)}\in G_{e_i}$ by construction, so $\Phi_i$ is a bijection between $H_i$ and the natural basis $\{sE_{e,f}:s\in G_{e_i}\up{ and }e,f\in X_i\}$ of $M_{r_i}(\C G_{e_i})$. 
Note further that if $e\in X_i$, then $\Phi_i(\ld e\rd) = e_i E_{e,e}.$ That is, $\Phi_i$ maps $\ld e \rd$ to the matrix which contains the identity of $G_{e_i}$ in the $e,e$ position, and hence $\Phi_i$ restricts to a bijection between $\{\ld e \rd : e\in X_i\}$ and the set of idempotents of the natural basis of $M_{r_i}(\C G_{e_i})$.

Since $\C S=\bigoplus_{i=1}^k\C H_i$, we may glue the $\Phi_i$ together to obtain an isomorphism
\[
\Phi: \C S \rightarrow \bigoplus_{i=1}^k M_{r_i}(\C G_{e_i}).
\]

By hypothesis we have $G_{e_i} \cong G_i$, so let $\omega_i:G_{e_i}\rightarrow G_i$ be an isomorphism. Extend $\omega_i$ to an isomorphism $\omega_i: M_{r_i}(\C G_{e_i}) \rightarrow M_{r_i}(\C G_{i})$ by declaring $\omega_i(gE_{e,f}) = \omega_i(g)E_{e,f}$ and extending linearly.
Glue the $\omega_i$ together to obtain an isomorphism 
\[
\Omega: \bigoplus_{i=1}^k M_{r_i}(\C G_{e_i}) \rightarrow \bigoplus_{i=1}^k M_{r_i}(\C G_{i}). 
\]
Then
\begin{equation*}
\Omega\circ\Phi : \C S \rightarrow \bigoplus_{i=1}^k M_{r_i}(\C G_{i})
\end{equation*}
is an isomorphism.

Let $B$ and $C$ denote the natural bases of $\bigoplus_{i=1}^k M_{r_i}(\C G_i)$ and $\bigoplus_{i=1}^k M_{r_i}(\C G_{e_i})$, respectively. $\Phi$ restricts to a bijection between the groupoid basis of $\C S$ and $C$ and $\Omega$ restricts to a bijection between $C$ and $B$, so we may use $\Omega \circ \Phi$ to define a partial order $\leq_{\Omega\circ\Phi}$ on $B$: for $b_1,b_2\in B$, let
\[
b_1\leq_{\Omega\circ\Phi} b_2 \iff (\Omega\circ\Phi)^{-1}(b_1) \leq (\Omega\circ\Phi)^{-1}(b_2).
\]
We now show that $\leq_{\Omega\circ\Phi}$ is a partial order on $B$ satisfying the hypotheses of Theorem \ref{ThmESN} and that $S$ is recoverable up to isomorphism from the construction of Theorem \ref{ThmESN} applied to $(B,\leq_{\Omega\circ\Phi})$.

Let $E(B)$ and $E(C)$ denote the set of idempotents of $B$ and $C$, respectively. $\Phi$ restricts to a bijection between $\{\ld e \rd : e\in E(S)\}$ and $E(C)$, and $\Omega$ restricts to a bijection between $E(C)$ and $E(B)$. From the definition of $\Omega$ it follows that the semilattice $(\{\ld e \rd : e\in E(S)\},\leq)$ is isomorphic to $(E(B),\leq_{\Omega\circ\Phi})$ by $\ld e \rd \mapsto \Omega\circ\Phi(\ld e\rd)$. In particular $(E(B),\leq_{\Omega\circ\Phi})$ is a meet-semilattice.

Note that for $b\in B$ and $s\in S$, if $b=\Omega\circ\Phi(\ld s\rd)$ then $b^{-1} = \Omega\circ\Phi(\ld s^{-1} \rd)$. From this and parts (\ref{simple10})--(\ref{simple11}) of Theorem \ref{ThmBasicProps} it is straightforward to check that $\leq_{\Omega\circ\Phi}$ satisfies hypotheses (\ref{ESNcondition2})--(\ref{ESNfinalcondition}) of Theorem \ref{ThmESN}.

By \eqref{eqZetaTransform}, we can recover $S$ up to isomorphism from $\leq_{\Omega\circ\Phi}$ and the multiplication of $B$. In particular, for $b\in B$, if we let
\begin{equation*}
\overline b = \sum_{a\in B: a\leq_{\Omega\circ\Phi} b} a,
\end{equation*}
then $\{\overline b : b\in B\}$ is an inverse semigroup isomorphic to $S$.

Finally, let $\sqsubseteq$ be any partial order on $E(B)$ for which $(E(S),\leq) \cong (E(B),\sqsubseteq)$. Write $()$ for the identity of any group. Let $\phi:E(S)\rightarrow E(S)$ be the function for which we have, for $e\in E(S)$,
\[
e \mapsto ()_{\phi(e),\phi(e)}
\]
in this isomorphism. Define $\gamma:B \rightarrow B$ by
\[
\gamma(g_{a,b}) = g_{\phi(a),\phi(b)}
\]
and define $\sqsubseteq'$ on $B$ by
\[
g_{\phi(a),\phi(b)} \sqsubseteq' h_{\phi(c),\phi(d)} \iff g_{a,b} \leq_{\Omega\circ\Phi} h_{c,d}.
\]
It is then straightforward to verify that $\gamma$ is a bijective operation-preserving map, that $(E(B),\sqsubseteq') = (E(B),\sqsubseteq)$, and that $(B,\sqsubseteq')$ is a poset isomorphic to $(B,\leq_{\Omega\circ\Phi})$. It follows that $\sqsubseteq'$ is a partial order on $B$ which restricts to $\sqsubseteq$ on $B$, meets the hypotheses of Theorem \ref{ThmESN}, and yields an inverse semigroup isomorphic to $S$ from the construction of Theorem \ref{ThmESN}.
\end{proof}

\subsection{\texttt{GPosets}}
\label{SecGPosetsCorrect}

In this section we prove the correctness of the implementation of $\texttt{GPosets}$ described in Section \ref{SecPartialOrders}. We require a sequence of lemmas.

\begin{lem}If $S$ is a finite inverse semigroup and $s,t\in S$ with $s\D t$ and $s\leq t$ then $s=t$.
\label{LemNoLEWithinADClass}
\end{lem}
\begin{proof}
Suppose $s\D t$ and $s\leq t$. Then $s^{-1}\leq t^{-1}$ so $\dom(s)=s^{-1}s\leq \dom(t)=t^{-1}t$.  Since $s\D t$, by part \eqref{simplecombo} of Theorem \ref{ThmBasicProps} we have $\dom(s)\D \dom(t)$. Thus by Theorem \ref{ThmLegitPartition} we have 
\[
|\{x\in S: x\leq \dom(s) \up{ and } x\D \dom(s)\}| = |\{x\in S: x\leq \dom(t) \up{ and }x\D \dom(s)\}|.
\]
If we were to have $\dom(s)< \dom(t)$, then by transitivity the quantity on the right would be strictly larger than the quantity on the left. Therefore we must have $\dom(s)=\dom(t)$. Then since $s\leq t$ we have $s=ts^{-1}s=tt^{-1}t=t$.
\end{proof}

We emphasize that the condition that $S$ be finite is necessary for Lemma \ref{LemNoLEWithinADClass}, as there exist infinite inverse semigroups for which the natural partial order does not reduce to equality on $\D$-classes.

Now let notation be as in Section \ref{SecPartialOrders}. 

\begin{lem}
\label{LemOrderingOfBis}Write $\leq$ for $\leq_{E(B)}$. 
If there exist idempotents $e\in B_i$ and $f\in B_j$ such that $f\leq e$, then $B_j \preceq B_i$.
\end{lem}
\begin{proof}
Suppose there are idempotents $e\in B_i$, $f\in B_j$ with $f\leq e$. 
If $e=f$ then since $B_i$ and $B_j$ cannot overlap nontrivially we have $B_i=B_j$. So suppose $e<f$ and for the sake of contradiction suppose $B_j \npreceq B_i$. Then $B_i \prec B_j$, so we have
\[
I(B_i) = \max(r\in \Z:B_i \cap L_r) \geq \max(r\in \Z:B_j \cap L_r) = I(B_j).
\]
Let $r_i=I(B_i)$ and let $\overline e \in B_i$ with $\overline e=()_{a,a}$ for some $a \in L_{r_i}$. Then, since $P$ is a $\D$-partition of $E$, we have
\[
|\{h < \overline e: h\in B_j\}| = |\{h < e:h\in B_j\}|.
\]
Since $f<e$ the quantity on the right is positive, so there exists $\overline h\in B_j$ with $\overline h < \overline e$. Therefore
\[
\max(r\in \Z: B_j \cap L_r) > r_i = I(B_i),
\]
a contradiction.
\end{proof}

Parts \eqref{LemRestrictingTheOrderPt2}--\eqref{LemRestrictingTheOrderPt5} of the following lemma concern the membership of elements in $\widehat B_i$ in restrictions to the $\widehat B_i$ of the partial orders on $B$ we seek. Part \eqref{LemRestrictingTheOrderPt6} is a technical result that will be used in the proof of Lemma \ref{LemGPosetsChildrenProperties}.

\begin{lem}
\label{LemRestrictingTheOrder}
Suppose $\leq$ is a partial order on $B$ satisfying the hypotheses of Theorem \ref{ThmESN} and  $1\leq i \leq k$. Then:
\begin{enumerate}[(i)]
	\item \label{LemRestrictingTheOrderPt2} $\forall t \in \widehat B_i$, if $s\leq t$ then $s^{-1}\leq t^{-1}$ and $s,s^{-1}\in \widehat B_i$. 
	\item \label{LemRestrictingTheOrderPt3} $\forall y,z\in \widehat B_i$, if $s\leq y$, $t\leq z$, $st\neq 0$, and $yz\neq 0$, then $st\leq yz$ and $s,t, s t, y z \in \widehat B_i$.
	\item \label{LemRestrictingTheOrderPt4} $\forall s\in \widehat B_i$, if $e\leq \dom(s)$, then $\exists ! t\leq s$ such that $\dom(t)=e$. We also have $t\in \widehat B_i$.
	\item \label{LemRestrictingTheOrderPt5} $\forall s\in \widehat B_i$, if $e\leq \ran(s)$, then $\exists ! t\leq s$ such that $\ran(t)=e$. We also have $t\in \widehat B_i$.
 \item \label{LemRestrictingTheOrderPt6} \label{LemNoNewCovers} If $h,h'\leq i$, $b \in B_h$, $b'\in B_{h'}$, and $b$ covers $b'$, then there exist idempotents $e\in B_h$, $f\in B_{h'}$, such that $e$ covers $f$. Furthermore $h'<h$.
\end{enumerate}
\end{lem}

\begin{proof}
To show \eqref{LemRestrictingTheOrderPt2}--\eqref{LemRestrictingTheOrderPt5} we only need to establish the claimed membership in $\widehat B_i$. It follows from hypotheses \eqref{ESNcondition2} and \eqref{ESNcondition3} of Theorem \ref{ThmESN} that if $s,t\in B$ with $s\leq t$, then $\ran(s)\leq \ran(t)$. The claimed membership in $\widehat B_i$ in \eqref{LemRestrictingTheOrderPt2}--\eqref{LemRestrictingTheOrderPt5} follows from Lemma \ref{LemOrderingOfBis}.

To show \eqref{LemRestrictingTheOrderPt6}, first suppose to the contrary that there exists $b\in B_h$ which covers $b'\in B_{h'}$, and for all idempotents $e\in B_h$, $f\in B_{h'}$, $e$ does not cover $f$. Then $\ran(b)$ does not cover any idempotents in $B_{h'}$. From the hypotheses of Theorem \ref{ThmESN} it follows that $\ran(b') < \ran(b)$. Since $\ran(b)$ does not cover $\ran(b')$, there exists an idempotent $f'\in B_r$ for some $r\leq i$ for which $\ran(b')< f' < \ran(b)$. 
By \eqref{LemRestrictingTheOrderPt5}, then, there exists $x\in B_r$ such that $x< b$ and $\ran(x)=f'$. Then, since $\ran(b')\in B_{h'}$ and $\ran(b')< f'$, there exists $y<x$ such that $y\in B_{h'}$ and $\ran(y)=\ran(b')$. In addition, there is a unique element $u\in B_{h'}$ such that $u<b$ and $\ran(u)=\ran(b')$. Since $b',y\in B_{h'}$, $b'<b$, $y<b$, and $\ran(y)=\ran(b')$, we have that $b'=y$. But then we have $b'=y<x<b$, contradicting the assumption that $b$ covers $b'$.
The final statement of \eqref{LemNoNewCovers} follows from Lemmas \ref{LemNoLEWithinADClass} and \ref{LemOrderingOfBis}.
\end{proof}

%
%

\begin{lem}
\label{LemGPosetsChildrenProperties}
Let $N = (\widehat B_i,\leq_{i},i)$ be a node of the search tree for $\texttt{GPosets}$ with $1\leq i<k$. Then the children of $N$ (produced by Algorithm \ref{AlgGPosetsChildren}) consist precisely of all possible nodes $(\widehat B_{i+1},\leq_{i+1},i+1)$ such that
\begin{enumerate}
	\item[(C1)]\label{Childrenitem1} if $a,b \in \widehat B_{i}$, then $a \leq_i b$ if and only if $a \leq_{i+1} b$,
	\item[(C2)]\label{Childrenitem2} $\forall s, t \in \widehat B_{i+1}$, if $s\leq_{i+1} t$ then $s^{-1}\leq_{i+1} t^{-1}$,
	\item[(C3)]\label{Childrenitem3} $\forall s,t,y,z\in \widehat B_{i+1}$, if $s\leq_{i+1} y$, $t\leq_{i+1} z$, $st\neq 0$, and $yz\neq 0$, then $st\leq_{i+1} yz$,
	\item[(C4)]\label{Childrenitem4} $\forall e,s\in \widehat B_{i+1}$, if $e\leq_{i+1} \dom(s)$, then $\exists ! t\leq_{i+1} s$ such that $\dom(t)=e$, and
	\item[(C5)]\label{Childrenitem5} $\forall e,s\in \widehat B_{i+1}$, if $e\leq_{i+1} \ran(s)$, then $\exists ! t\leq_{i+1} s$ such that $\ran(t)=e$.
\end{enumerate}
\end{lem}

\begin{proof}
Suppose $\leq_{i+1}$ is a partial order on $\widehat B_{i+1}$ satisfying (C1)--(C5). By part \eqref{LemNoNewCovers} of Lemma \ref{LemRestrictingTheOrder}, if $b$ covers $b'$ in $\leq_{i+1}$ and $b\in B_{i+1}$, then $b'\in B_j$ for some $j \leq i$ and for which $B_{i+1}$ covers $B_j$. Furthermore, the construction of Theorem \ref{ThmESN} applied to $(\widehat B_{i+1},\leq_{i+1})$ would produce a finite inverse semigroup, so by Lemma \ref{LemNoLEWithinADClass} we have that $\leq_{i+1}$ restricts to equality on $B_h$, $\forall h\leq i+1$. Finally, by Proposition \ref{PropLegitPartition}, if $h \leq i$, then for all $s,t\in B_{i+1}$, $|\{s'\in B_h: s'\leq_{i+1} s\}| = |\{t'\in B_h: t'\leq_{i+1} t\}|$. Therefore every partial order $\leq_{i+1}$ on $\widehat B_{i+1}$ satisfying (C1)--(C5) appears among the children of $N$. 

We now show conversely that every child of $N$ satisfies (C1)--(C5). (C1) is satisfied by construction (in particular, by the specification of the function $\tt{PosetPossibilities}$). Let $(\widehat B_{i+1}, \leq_{i+1}, i+1)$ be a child of $N$, so $\leq_{i+1}$ is a partial order on $\widehat B_{i+1}$ such that 
\begin{enumerate}[(i)]
	\item\label{BigProp1} $\forall s,t\in \widehat B_i$, if $s\leq_{i+1} t$ then $s^{-1} \leq_{i+1} t^{-1}$,
	\item\label{BigProp2} $\forall s,t,y,z\in \widehat B_i$, if $s\leq_{i+1} y$, $t\leq_{i+1} z$, $s t\neq 0$, and $y z\neq 0$, then $s t \leq_{i+1} y z$, 
	\item\label{BigProp3} $\forall e,s\in \widehat B_i$, if $e\leq_{i+1} \dom(s)$, then $\exists! t\leq_{i+1} s$ such that $\dom(t)=e$,
	\item\label{BigProp4} $\forall e,s\in \widehat B_i$, if $e\leq_{i+1} \ran(s)$, then $\exists! t\leq_{i+1} s$ such that $\ran(t)=e$,
	\item\label{BigProp6} if $s\in \widehat B_{i+1}$, then $\forall h\in \{1 ,\ldots, i\}$, $|\{t\in B_h:t\leq_{i+1} s\}| = |\{e\in B_h: e\leq_{i+1} \ran(s)\}|=|\{e\in B_h: e\leq_{i+1} \dom(s)\}|$.

\noindent Furthermore, if $B_{i+1}$ covers $B_j$, then 

		\item\label{BigProp8} $\forall t\in B_{i+1}$, if $s\in B_j$ with $s\leq_{i+1} t$, then $s^{-1}\leq_{i+1} t^{-1}$,
		\item\label{BigProp9} $\forall y,z\in B_{i+1}$, if $s\leq_{i+1} y$ and $t\leq_{i+1} z$ with $s,t\in B_j$, $s t \neq 0$, and $y z\neq 0$, then $s t\leq_{i+1} yz$,
		\item\label{BigProp10} $\forall s\in B_{i+1}$, if $e\leq_{i+1} \dom(s)$ with $e\in B_j$, then $\exists! t\leq_{i+1} s$ such that $\dom(t)=e$, and
		\item\label{BigProp11} $\forall s\in B_{i+1}$, if $e\leq_{i+1} \ran(s)$ with $e\in B_j$, then $\exists! t\leq_{i+1} s$ such that $\ran(t)=e$.
\end{enumerate}

In order, we explain why $(\widehat B_{i+1},\leq_{i+1}, i+1)$ satisfies (C2), (C4), (C5), and (C3). For the remainder of the proof, write $\leq$ for $\leq_{i+1}$.

\begin{enumerate}
	\item[(C2)] Suppose $s,t\in \widehat B_{i+1}$ and $s\leq t$. If $t\in \widehat B_{i}$ or $t=s$ we are done, so suppose $t\in B_{i+1}$ and $s<t$. Then $s\in B_h$ for some $h\leq i$. If $B_{i+1}$ covers $B_h$ we are done, so suppose $B_{i+1}$ does not cover $B_h$. Let $j\leq i$ and $t'\in B_j$ such that $B_{i+1}$ covers $B_j$ and $s\leq t'\leq t$. Then $t'^{-1}\leq t$ (by \eqref{BigProp8}) and $s^{-1}\leq t'^{-1}$ (by \eqref{BigProp1}), so $s^{-1}\leq t^{-1}.$
	\item[(C4)] Let $e,s\in \widehat B_{i+1}$ and $e\leq \dom(s)$. If $s\in \widehat B_{i}$ or $e\in B_{i+1}$ we are done, so suppose $s\in B_{i+1}$ and $e\in B_h$ for some $h\leq i$. If $B_{i+1}$ covers $B_h$ we are done, so suppose $B_{i+1}$ does not cover $B_h$. Let $j\leq i$ such that $B_{i+1}$ covers $B_j$ and there exists $f\in B_j$ such that $e\leq f\leq \dom(s)$. Then $\exists s'\in B_j$ such that $s'\leq s$ and $\dom(s)=f$ (by \eqref{BigProp10}) and $\exists t\in B_h$ such that $t\leq s'$ and $\dom(t)=e$ (by \eqref{BigProp3}), so $\exists t\in B_h$ such that $t\leq s$ and $\dom(t)=e$. The uniqueness of $t$ follows from \eqref{BigProp6}---in particular, if $\{e\in B_h:e\leq \dom(s)\} = \{e_1 ,\ldots, e_p\}$, then for $1\leq q\leq p$, let $\phi(e_q)=\{t\in B_h:t\leq s, \dom(t)=e_q\}$. By the preceding argument, $|\phi(e_q)|\geq 1$ for all $1\leq q \leq p$. Since the $e_q$ are distinct, the $\phi(e_q)$ do not overlap. Furthermore $\cup_{q=1}^p \phi(e_q) = \{t\in B_h: t\leq s, \dom(t)\leq \dom(s)\}$, so $ |\{t\in B_h: t\leq s,\dom(t)\leq\dom(s)\}|\geq p$. But by \eqref{BigProp6}, $|\{t\in B_h: t\leq s, \dom(t)\leq\dom(s)\}| \leq p$, so $|\{t\in B_h: t\leq s,\dom(t)\leq \dom(s)\}| = p$ and $|\phi(e_q)| = 1$ for all $1\leq q \leq p$.
	
	\item[(C5)] Similar to (C4).

	\item[(C3)] Let $s,t,y,z\in \widehat B_{i+1}$, $s\leq y$, $t\leq z$, $s t \neq 0$, and $y z\neq 0$. We need to show $s t\leq yz$. Since $y z\neq 0$, we have $y,z\in B_x$ for some $x\leq i+1$. If $x\leq i$ we are done so suppose $y,z\in B_{i+1}$. Similarly, we have $s,t\in B_h$ for some $h\leq i+1$. If $h=i+1$ or $B_{i+1}$ covers $B_h$ we are done, so suppose $h\leq i$ and $B_{i+1}$ does not cover $B_h$. 	By hypothesis we have $\dom(s)=\ran(t)$ and $\dom(y)=\ran(z)$. 
	Let $j\leq i$ be such that $B_{i+1}$ covers $B_j$ and there exists $y'\in B_j$ such that $s\leq y'\leq y$. It follows from \eqref{BigProp8} and \eqref{BigProp9} that $\dom(y')\leq \dom(y)$, and from \eqref{BigProp1} and \eqref{BigProp2} that $\dom(s)\leq \dom(y')$. 
	Since $\dom(y')\leq \dom(y)=\ran(z)$, by \eqref{BigProp11} there is a unique $z'\in B_j$ such that $z'\leq z$ and $\ran(z')=\dom(y')$. By \eqref{BigProp9} we have $y'z'\leq yz$. 
	We claim that $t\leq z'$, for if not there would exist $t'\in B_h$ such that $t\neq t'$, $t'\leq z'$, and $\ran(t')=\dom(s)$; but then we would have distinct elements $t,t'\in B_h$ with $t,t'\leq z$ and $\ran(t)=\ran(t')=\dom(s)$, contradicting (C5). So $t\leq z'$. By \eqref{BigProp2}, then, $s t\leq y'z'$, so by transitivity we have $s t \leq y z$.
\end{enumerate}
\end{proof}

\begin{prop}[Correctness of implementation of $\tt{GPosets}$]The set of partial orders in the leaves of the search tree for $\tt{GPosets}$ is precisely the set of partial orders on $B=\widehat B_k$ satisfying the hypotheses of Theorem \ref{ThmESN}.
\end{prop}
\begin{proof}
Suppose $\leq$ is a partial order on $B$ satisfying the hypotheses of Theorem \ref{ThmESN}.
By part \eqref{LemNoNewCovers} of Lemma \ref{LemRestrictingTheOrder}, the restriction of $\leq$ to $B_1$ is $\leq_{E(B)}$. Lemmas \ref{LemRestrictingTheOrder}--\ref{LemGPosetsChildrenProperties} and induction establish that $\leq$ may be found as one of the leaves of the search tree.
Conversely, Lemma \ref{LemGPosetsChildrenProperties} establishes that the partial order $\leq_k$ of every leaf $(\widehat B_k,\leq_k,k)$ of the search tree satisfies the hypotheses of Theorem \ref{ThmESN}.
\end{proof}

\section{Acknowledgments}

We are grateful to Sam Houston State University (SHSU) and the IT\textsf{@}Sam department at SHSU for their assistance in building and maintaining the server on which we obtained our computational results. We are also thankful to the anonymous referee, whose suggestions have helped improve the presentation and readability of the paper.

\clearpage

\bibliographystyle{plain}
\bibliography{Enumbib}
\end{document}